\documentclass{amsart}
\usepackage{amsmath,amsthm,amssymb,mathtools,mathdots,nicefrac,tikz,enumitem,url}

\usepackage{algorithm}
\usepackage[noend]{algpseudocode}
\usepackage{tikz-cd}
\usepackage{dirtytalk}
\usepackage{relsize}
\usepackage{stmaryrd}
\usepackage{soul}

\theoremstyle{definition} 
\newtheorem{theorem}{Theorem}[section]
\newtheorem{corollary}[theorem]{Corollary}
\newtheorem{lemma}[theorem]{Lemma}
\newtheorem{proposition}[theorem]{Proposition}
\newtheorem{definition}[theorem]{Definition}

\newtheorem{conjecture}[theorem]{Conjecture}

\newcommand{\R}{\mathbb{R}}
\newcommand{\Z}{\mathbb{Z}}

\newcommand{\cS}{\mathcal{S}}

\newcommand{\cE}{\mathcal{E}}
\newcommand{\cH}{\mathcal{H}}

\newcommand{\sn}{\mathrm{sn}}

\DeclareMathOperator{\gon}{gon}

\DeclareMathOperator{\lw}{lw}
\DeclareMathOperator{\ls}{ls_{\square}}

\definecolor{CornflowerBlue}{rgb}{0.39, 0.6, 0.99}

\definecolor{Orange}{rgb}{0.9, 0.7, 0.4}

\definecolor{Fuchsia}{rgb}{1, 0.2, 0.7}

\usepackage{natbib}
\usepackage{mathrsfs}
\providecommand{\scr}{\mathscr} 
\usepackage{hyperref}
    \hypersetup{
        colorlinks=true,
        citecolor=blue, 
        urlcolor=black, 
        filecolor=magenta, 
        urlcolor=magenta, 
        linkcolor=purple,
        linkbordercolor={1 0 0}
    }
\usepackage{todonotes}

\setuptodonotes{color=orange!30}

\def\floor#1{\left\lfloor {#1} \right \rfloor}
\def\ceil#1{\left\lceil {#1} \right\rceil}

\def\abs#1{\left|{#1}\right|}

\definecolor{Ralph}{rgb}{0.80, 0.39, 0.70}

\newcommand{\mnd}[2]{\mathbb{M}^{\textrm{nd}}_{#1, #2}}
\newcommand{\mndexp}[2]{\mathbb{M}^{\textrm{nd}}_{#1, \underline{#2}}}
\newcommand{\bone}[2]{U(#1,#2)}
\newcommand{\modspace}[1]{\mathbb{M}_{#1}}
\newcommand{\cR}{\mathcal{R}}

\DeclareMathOperator{\egon}{egon}
\DeclareMathOperator{\conv}{conv}

\newcommand{\dmnd}[2]{\dim\Bigl(\mathbb{M}^{\mathrm{nd}}_{#1, #2}\Bigr)}
\newcommand{\dmndexp}[2]{\dim\left(\mathbb{M}^{\mathrm{nd}}_{#1, \underline{#2}}\right)}
\newcommand{\dmndc}[2]{\dim\bigl(\mathbb{M}^{\mathrm{nd}}_{#1, #2}\bigr)}
\newcommand{\dmndexpc}[2]{\dim\bigl(\mathbb{M}^{\mathrm{nd}}_{#1, \underline{#2}}\bigr)}
\newcommand{\dmodspace}[1]{\dim\left(\mathbb{M}_{#1}\right)}

\graphicspath{{figures/}}

\newif\ifarxiv
\arxivtrue

\title{The $d$-gonal locus in the moduli space of tropical plane curves}
\author{Desmond Leitz}
\address{Stony Brook University}
\email{desmond.leitz@stonybrook.edu}

\author{Ralph Morrison}
\address{Williams College}
\email{TODO}

\author{S\o ren Newman-Taylor}
\address{Williams College}
\email{sorenfnt@gmail.com}

\author{Vincent X. Wang}
\address{University of Chicago}
\email{vincentxwang@uchicago.edu}
\date{}

\begin{document}

\begin{abstract}
We introduce and study the locus $\mnd{g}{d} $ of genus $g$ tropical plane curves of gonality $d$ inside the moduli space $\mathbb{M}^{\textrm{nd}}_{g}$ of tropical plane curves of genus $g$. Each such tropical curve arises from a Newton polygon, and we conjecture that the gonality of the tropical curve is equal to an easily computed parameter of this polygon called the expected gonality, closely related to the lattice width of the polygon. Let $\mndexp{g}{d}$ denote the locus of tropical curves whose associated Newton polygon has expected gonality $d$. We prove that for fixed $d$ and sufficiently large genus $g$, the dimensions of these two loci agree:
\[
\dmnd{g}{d} = \dmndexp{g}{d}.
\]
Our results provide evidence that, in sufficiently high genus compared to expected gonality, the gonality of a tropical curve is determined by the expected gonality of the Newton polygon from which it arises.  We also leverage our results to prove bounds on the dimensions of moduli spaces of non-degenerate algebraic curves.
\end{abstract}

\maketitle

\section{Introduction}
\label{sec:introduction}
Tropical geometry is a tool for studying objects in algebraic geometry through a combinatorial lens.  In the one-dimensional case, an algebraic curve $C$ is associated to a metric graph $\Gamma$.  Using Baker and Norine's \cite{bn} graph-theoretic analog of divisor theory, we can study invariants such as the (divisorial) gonality of $\Gamma$, which provides information on the gonality of $C$ (the minimal degree of a map from a curve to the projective line). In particular, we have $\gon(\Gamma)\leq \gon(C)$ \cite{baker}.

On a moduli-theoretic level, we stratify curves by their genus to study $\mathcal{M}_g$, the $(3g-3)$-dimensional moduli space of algebraic curves of genus $g$, whose points correspond to such curves. The graph-theoretic analog is $\mathbb{M}_g$, the $(3g-3)$-dimensional moduli space of metric graphs of genus\footnote{That is, the first Betti number.} $g$, whose points correspond to such metric graphs.

The correspondence between curves and graphs is not without complications, whether directly associating curves with graphs one at a time or through the lens of moduli spaces. A good example to illustrate this is the so-called ``lollipop'' graph of genus $3$. This is a graph with divisorial gonality $2$ regardless of its metric, but only arises from algebraic curves of gonality $3$. Thus, letting $\Gamma$ be a metric lollipop graph and $C$ be a curve associated to it, we have $\gon(C)=3$ and $\gon(\Gamma)=2$.  Moreover, since the graph has $6$ edges, there exists a $6$-dimensional family hyperelliptic graphs of genus $3$, which does not exist in the algebraic world.  If we let $\mathcal{M}_{g,d}$ (respectively  $\mathbb{M}_{g,d}$) denote the moduli space of algebraic curves (respectively metric graphs) of genus $g$ and divisorial gonality $d$, then we have $\dim(\mathcal{M}_{3,2})=5$ and $\dim(\mathbb{M}_{3,2})=6$.

In summary, we have both of the following unfortunate phenomena:
\begin{enumerate}
    \item It is \textbf{not} always the case that $\gon(\Gamma)= \gon(C)$, even if $\Gamma$ is trivalent.
    \item For fixed genus $g$ and gonality $d$, it is \textbf{not} the case that the algebraic and tropical moduli spaces always have the same dimension; that is, it is not always the case that
    \[\dim(\mathcal{M}_{g,d})=\dim(\mathbb{M}_{g,d}).\]
\end{enumerate}

It is natural to ask if there are other frameworks where these problems cannot occur.  For instance, Chan proved in 2013 that when restricting to $2$-edge-connected hyperelliptic tropical curves of genus $g$, the moduli space has the expected dimension of $2g-1$ \cite[Theorem 1.4]{cha13}.  Another avenue is presented by Cools and Draisma in \cite{CD18}, who replace divisorial gonality with \emph{tree gonality}, the minimum degree of a tropical morphism from any modification of the metric graph to a tree.  In this setting, they prove that the space of graphs of genus $g$ and gonality at most $d$ matches the dimension of the corresponding space of algebraic curves.

In this work, we search for another framework where these two phenomena cannot occur: namely that of non-degenerate plane curves on the algebraic side, and smooth tropical plane curves on the tropical side. Here every curve (algebraic or tropical) arises from a Newton polygon with $g$ interior lattice points, where $g$ is the genus of the curve.  On the algebraic side, we consider the spaces $\mathcal{M}_P$ and $\mathcal{M}_g^{\textrm{nd}}$ introduced in \cite{castryck_voight}.  The first is the moduli space of all non-degenerate algebraic curves with Newton polygon $P$; and the second is the union of all such moduli spaces over all polygons $P$ of genus $g$.  The corresponding tropical moduli spaces are $\mathbb{M}_P$ and $\mathbb{M}_g^{\textrm{nd}}$, first studied in \cite{brodsky_joswig_morrison_sturmfels} (though they denoted the latter space by $\mathbb{M}_g^{\textrm{planar}}$).  These spaces are closely related; for instance, $\dim(\mathcal{M}_P)=\dim(\mathbb{M}_P)$ \cite{fixed_newton_polyon}.  

Stratifying the algebraic space by gonality is straightforward.  Castryck and Cools \cite{CC17} showed that the gonality of an algebraic curve can be computed from its Newton polygon.  For a polygon $P$, we define its \emph{expected gonality}, denoted by $\egon(P)$, to be the gonality of an algebraic curve with Newton polygon $P$. The value of $\egon(P)$ can be derived from the lattice width of $P$ (see Section~\ref{sec:background} for a full characterization). For polygons $P$ of genus $5$ or greater\footnote{In fact, we can say this for polygons of genus $2$ or greater with a single exception: the genus $4$ triangle $2\Upsilon$ illustrated in Figure \ref{figure:triangle_families}, which has expected gonality  $3$ even though its interior polygon has lattice width $2$.}, the expected gonality is equal to two more than the lattice width of the interior polygon of $P$.   Thus we may refine $\mathcal{M}_g^{\textrm{nd}}$ as $\mathcal{M}_{g,d}^{\textrm{nd}}$ by restricting to the curves of gonality $d$, or equivalently to the polygons with expected gonality $d$.

The tropical story is more intricate. For given $g$ and $d$, we introduce in this paper two loci within $\modspace{g}^{\text{nd}}$, stratified either by the divisorial gonality of the metric graph, or by the expected gonality of the associated Newton polygon. Specifically, we define the \emph{gonality locus} $\mnd{g}{d}$ and the \emph{expected gonality locus} $\mndexp{g}{d}$:
\[\mathbb{M}_{g,d}^{\textrm{nd}}=\{\Gamma\in \mathbb{M}_g^{\textrm{nd}}\,|\, \gon(\Gamma)=d\}.\]
\[\mathbb{M}_{g,\underline{d}}^{\textrm{nd}}=\!\!\!\!\!\!\!\!\bigcup_{\stackrel{g(P)=g}{\textrm{egon}(P)=d}} \!\!\!\!\!\!\!\mathbb{M}_P.\]
Note that we have $\dim(\mathcal{M}^\textrm{nd}_{g,d})=\dim(\mathbb{M}^\textrm{nd}_{g,\underline{d}})$, since each is a finite union of $\mathcal{M}_P$ or $\mathbb{M}_P$ over the same set of polygons.

Having established this framework, we pose the following two questions:

\begin{enumerate}
    \item If $\Gamma$ is a smooth plane tropical curve arising from a non-degenerate algebraic curve $C$, is it always the case that $\gon(\Gamma)= \gon(C)$?
    \item Fixing genus as $g$ and gonality as $d$, do the algebraic and tropical moduli spaces always have the same dimension? That is, do we have
    \[\dim\left(\mathcal{M}^\textrm{nd}_{g,d}\right)=\dmnd{g}{d},\]
    or equivalently 
    \[
\dmndexp{g}{d} = \dmnd{g}{d}?
\]
\end{enumerate}

We conjecture that the answer to Question (1) is yes, or equivalently:

\begin{conjecture}\label{conj:gon=egon}
    Let $\Gamma$ denote any smooth tropical plane curve with Newton polygon $P$. Then $\gon(\Gamma) = \egon(P)$.
\end{conjecture}

The conjecture holds trivially for polygons of genus $0$ or $1$, and has been shown to be true for hyperelliptic tropical curves \cite{tropical_hyperelliptic_curves_in_the_plane}.   Note that an affirmative answer to this conjecture would yield an affirmative answer to Question (2).

Answering Question (1) is beyond the scope of this present paper, so we aim to answer Question (2), at least when $g$ is sufficiently large compared to $d$.  Our main theorem is the following:

\begin{theorem}\label{tmd:equaldim}
Fix $ d \geq 3 $. Then for any genus $g$ satisfying
\[g \geq \max\{d^3,32\},\]
the following equality holds:
\[
\dmnd{g}{d}=\dmndexp{g}{d}.
\]
\end{theorem}

In Section~\ref{sec:dimlowg}, we also verify this equality for certain low genera ($g \leq 6$ and $g = 8$). We conjecture that equality holds for all $g, d$, which is a weaker prediction than Conjecture~\ref{conj:gon=egon}:

\begin{conjecture}\label{conj:equalmodspace}
Let $g \geq 0$ and $d \geq 1$. Then,
\[\dmnd{g}{d}=\dmndexp{g}{d}.\]
\end{conjecture}

A key ancillary result of the proof of Theorem~\ref{tmd:equaldim} is the following upper bound on the dimension of a moduli space of tropical curves with a fixed Newton polygon $P$, denoted $\mathbb{M}_P$, which extends the bound in \cite{fixed_newton_polyon} of $2g + 1$ to incorporate combinatorial aspects of the Newton polygon.

\begin{theorem}
\label{lma:maxdim2_restate}
    Let $P$ be a maximal lattice polygon with $g$ interior points and expected gonality $d>1$. Then,
    \[\dim\left(\modspace{P}\right) \leq g + \frac{2g}{d-1} + 2d - 3.\]
\end{theorem}

This result translates nicely to bounds on the dimensions of algebraic moduli spaces.

\begin{corollary}
    Let $P$ be a maximal lattice polygon with $g$ interior points and expected gonality $d>1$. Then,
    \[\dim\left(\mathcal{M}_{P}\right) \leq g + \frac{2g}{d-1} + 2d - 3.\]
    Moreover, we have
    \[\dim\Bigl(\mathcal{M}_{g,d}^{\textrm{nd}}\Bigr)\leq g + \frac{2g}{d-1} + 2d - 3. \]
\end{corollary}

We obtain some other results regarding algebraic moduli spaces, which follow from the results in Section \ref{sec:dimmodspace_lowgd}. We note that $\mathcal{M}_{g,2}^{\textrm{nd}}$ is the locus of hyperelliptic curves, which is well-known to have dimension $2g-1$; and $\mathcal{M}_{g,3}^{\textrm{nd}}$ is the locus of trigonal curves, which is well-known to have dimension $2g+1$ for $g\geq 4$; see \cite[Proposition 1.3]{trigonal_reference} for a proof when $g\geq 5$.  Our new contributions are dimensions for $\mathcal{M}_{g,4}^{\textrm{nd}}$ for certain sufficiently large values of $g$, and estimates for $\mathcal{M}_{g,5}^{\textrm{nd}}$ for all sufficiently large values of $g$.

\begin{theorem}  Let $U(g,d)=g+\frac{2g}{d-1}+2d-3$. For $g\geq 7$ and $g\equiv 1\pmod 3$, we have
\[\dim \left(\mathcal{M}_{g,4}^\textrm{nd}\right)=\lfloor U(g,4)\rfloor.\]
For $g\geq 12$, we have
\[\lfloor U(g,5)\rfloor \geq \dim\left(\mathcal{M}_{g,5}^\textrm{nd}\right)\geq\lfloor U(g,5)\rfloor-1.\]
\end{theorem}

These formulas follow from Theorem~\ref{thm: d=4} and Lemma~\ref{lem:d=5LB}.

We remark that the primary historical interest in relating the gonality of a metric graph to the gonality of an algebraic curve is to use a known graph gonality to study an unknown curve gonality.  In our setting, the gonalities of the algebraic curves are well-understood.  Thus our work is less in the tradition of computing graph gonalities in service of computing those of algebraic curves, and more in the spirit of \cite[Theorem 1.4]{cha13}, aiming to understand and reconcile the discrepancies between tropical and algebraic curves and their respective moduli spaces.

\subsection{Organization of the paper}

In Section~\ref{sec:background}, we review background on lattice polygons, divisorial gonality of metric graphs, and tropical curves. In Section~\ref{sec:AG}, we construct the tropical moduli spaces analogous to their algebraic counterparts, and summarize some known results. The proof of our main result, Theorem~\ref{tmd:equaldim}, is spread across Sections~\ref{sec:lower},~\ref{sec:upper}, and ~\ref{sec:equal-dim}, with a proof of Theorem~\ref{lma:maxdim2_restate} in Section~\ref{sec:upper}. Finally, in Section~\ref{sec:dimmodspace_lowgd}, we show equality of the dimensions of the two moduli spaces in low genus, as well as explicit computations of the dimension of $\mndexp{g}{d}$ for select cases.

\ifarxiv 
\subsection{Acknowledgements}

This research was supported by the National Science Foundation, under NSF Grant DMS-2241623 and the Finnerty Fund. We thank an anonymous reviewer for many helpful comments and suggestions that improved the paper.

We acknowledge the use of generative AI tools to check the arguments and correctness, as well as proofreading for presentation issues. A generative AI tool also found a mistake in the proof of an equivalent statement to Proposition~\ref{lem:rPcPbound} in an earlier version of the manuscript, and sketched an similar alternative proof. This proof was simplified and written up by the authors, who assume responsibility for all content.
\fi

\section{Background and definitions}\label{sec:background}

The paper primarily uses tools in discrete geometry and divisor theory. In this section, we introduce the main objects and techniques, as well as several key definitions.

\subsection{Lattice polygons and subdivisions}
\label{sec:bg-lattice_polygons}

Throughout this paper, we take a \emph{lattice polygon} (or simply, a polygon) to be the convex hull of a non-empty finite set of points in $\Z^2\subset\R^2$. A polygon $P$ is \emph{maximal} if it is not properly contained in any other lattice polygon with the same interior polygon. A $\Z$-affine transformation of $\R^2$ is a map $\varphi: \R^2 \to \R^2$ of the form $p \mapsto Ap + b$ for some $A \in \operatorname{GL}_2(\Z)$ and $b \in \Z^2$. We say two polygons $P, P'$ are \emph{equivalent} if and only if there exists a $\Z$-affine transformation $\varphi$ such that $\varphi(P) = P'$.

We define two specific families of triangles here. For a positive integer $d$, let $d\Sigma$ denote the triangle $\operatorname{conv}\left\{(0,0), (d,0), (0,d)\right\}$, and $d\Upsilon$ denote the triangle $\operatorname{conv}\left\{(-d,-d), (d,0), (0,d)\right\}$. Examples are illustrated in Figure \ref{figure:triangle_families}.

    \begin{figure}[hbt]
        \centering
        \includegraphics[width=0.4\linewidth]{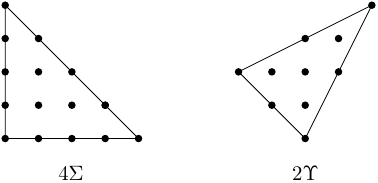}
        \caption{The triangles $4\Sigma$ and $2\Upsilon$.}
        \label{figure:triangle_families}
    \end{figure}

We use $A(P)$ to denote the \emph{area} of a polygon. Let the \emph{interior polygon} $P^{(1)}$ be the convex hull of the lattice points in the interior of $P$. The \emph{genus} of a polygon $P$, denoted $g(P)$, is the number of interior lattice points. A polygon is \emph{non-hyperelliptic} if $\dim(P^{(1)}) = 2$ and \emph{hyperelliptic} if $\dim(P^{(1)}) \leq 1$, or equivalently if the interior lattice points are all collinear. Denote by $\partial P$ the topological boundary of $P$ in $\mathbb{R}^2$.

An alternate description of a polygon $P$ is as the intersection of half-planes. Each edge $\tau \subset P$ corresponds to a half-plane
\[\cH_{\tau} := \{(x,y) \in \R^2: \alpha_\tau x + \beta_\tau y \leq \gamma_\tau\}\]
such that $P = \bigcap_{\tau}\cH_\tau$. The integer coefficients $\alpha_\tau, \beta_\tau, \gamma_\tau$ can be uniquely determined by requiring $\gcd(\alpha_\tau, \beta_\tau) = 1$. For each such half-plane, we define 
\[\cH_{\tau}^{(-1)} := \{(x,y) \in \R^2: \alpha_\tau x + \beta_\tau y \leq \gamma_\tau+1\}.\]

We call $P^{(-1)} := \bigcap_{\tau}\cH_\tau^{(-1)}$ the \emph{relaxed polygon} of $P$. The intersection of $P^{(-1)}$ and the boundary of $\cH_{\tau}^{(-1)}$ is denoted $\tau^{(-1)}$. As illustrated in the right of Figure~\ref{figure:pushout}, a relaxed polygon need not be a lattice polygon.  It turns out that the lattice polygons that are the relaxations of other lattice polygons are precisely the maximal non-hyperelliptic polygons:

\begin{proposition}[{\cite[\S 2.2]{Koelman_Dissertation} and \cite[Lemmas 9 and 10]{HS09}}]
    \label{prop:pushinpushout}
    A non-hyperelliptic polygon $P$ is maximal if and only if $P$ is the relaxed polygon of $P^{(1)}$; that is, $\left(P^{(1)}\right)^{(-1)} = P$.
\end{proposition}

    \begin{figure}[hbt]
        \centering
        \includegraphics[width=0.8\linewidth]{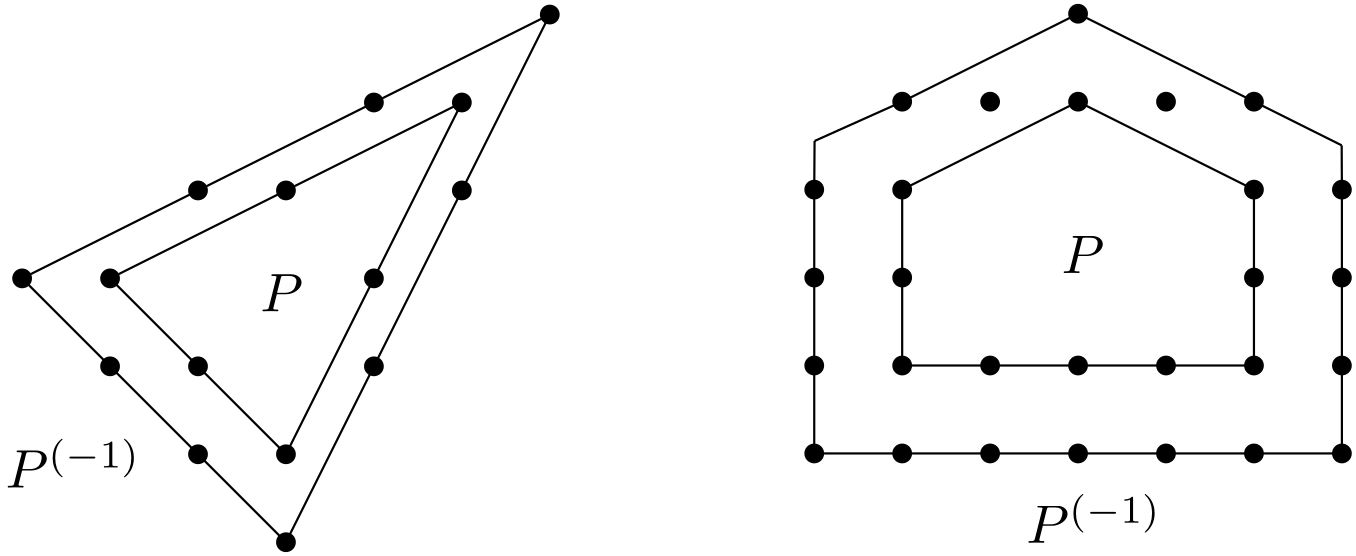}
        \caption{Polygons and their relaxations. Note that on the right, $P^{(-1)}$ is not a lattice polygon, meaning that $P$ is not the interior polygon of any lattice polygon.}
        \label{figure:pushout}
    \end{figure}

As introduced in \cite{fixed_newton_polyon}, \emph{column vectors} are tools used to compute the dimension of $\mathbb{M}_P$ (which was described in Section~\ref{sec:introduction}).

\begin{definition}
    A nonzero vector $v \in \Z^2$ is a \emph{column vector} of a polygon $P$ if there exists a facet $\tau \subset P$ such that
    \[v + \left((P \setminus \tau)\cap\Z^2\right) \subset P.\]
\end{definition}

Geometrically, $v$ is a column vector if the lattice points of $P$, excluding those on one of its edges $\tau$, remain inside $P$ after being translated by $v$. We use $c(P)$ to denote the number of column vectors of a polygon $P$. Two polygons and their column vectors are illustrated in Figure \ref{figure:columns}.

    \begin{figure}[hbt]
        \centering
        \includegraphics[width=0.8\linewidth]{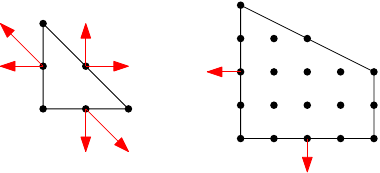}
        \caption{Polygons and their column vectors.}
        \label{figure:columns}
    \end{figure}

The \emph{lattice width} of a two-dimensional polygon $P$, denoted $\operatorname{lw}(P)$, is the smallest positive integer $s$ such that there exists a $\Z$-affine transformation mapping $P$ into the horizontal strip $\{(x, y) \in \R^2 : 0 \leq y \leq s\}$. Geometrically, this is one less than the minimum number of parallel lattice lines needed to cover the lattice points of $P$. Similarly, the \emph{lattice size} of a polygon $P$, denoted $\ls(P)$, is the smallest integer $s$ for which there exists a $\Z$-affine transformation that can map $P$ into $s\square$, defined as the square $\conv\{(0,0), (0,s), (s,s), (s,0)\}$.

For the purposes of this paper, a \emph{subdivision} of a polygon $P$ is a partition of $P$ into smaller polygons, called \emph{cells}. We refer the reader to \cite[Definition 2.3.1]{Triangulations} for a complete ``bullet-proof'' definition. A \emph{unimodular triangulation} is a subdivision where all cells are triangles of area 1/2. Equivalently, in two dimensions, a subdivision is a unimodular triangulation if and only if it cannot be further subdivided using cells whose vertices are lattice points \cite[Corollary 9.3.6]{Triangulations}. A subdivision $\cS$ is \emph{trivial} if it has a cell that is the entirety of $P$.

A function $\omega: P\cap\Z^2 \to \R$ is a \emph{height function}. A subdivision $\cS$ is \emph{regular} if it can be obtained by projecting the facets of the lower convex hull of the graph in $\R^3$ of some height function $\omega$. The notation $P\vert_C$ denotes the restriction of $P$ on some cell $C$, and a \emph{refinement} $\cR(P\vert_C, \omega)$ denotes the subdivision of $C$ induced by $\omega$.

\subsection{Expected gonality}

We now define the \emph{expected gonality} of $P$, denoted $\egon(P)$. This definition is motivated by a result from algebraic geometry by Castryck and Cools \cite[Corollary 6.2]{CC17}, which states that the gonality of an algebraic curve can be determined from the geometric properties of its Newton polygon's interior, $P^{(1)}$. 

The value of $\egon(P)$ depends on the geometry of $P^{(1)}$:
\begin{itemize}
    \item If $P^{(1)}$ is two-dimensional, then we define $\egon(P) = \lw(P^{(1)}) + 2$. There is one exception: if $P \cong 2\Upsilon$, then we set $\egon(P) = 3$.
    \item If $P^{(1)}$ is a degenerate line segment (or point), then we define $\egon(P) = 2$.
    \item If $P^{(1)}$ is empty, then we define $\egon(P) = 1$.
\end{itemize}

Notably, the last two cases are consistent with the main formula $\egon(P) = \lw(P^{(1)}) + 2$ if we formally define the lattice width of a degenerate line segment or point to be $0$, and that of an empty set to be $-1$. While the formal definition uses the interior polygon $P^{(1)}$, the expected gonality is almost always equal to the lattice width of the polygon $P$ itself:

\begin{proposition}[{\cite[Theorem 4]{CC12}}]\label{lwisegon}
    Let $P$ be a polygon. Then,
    \[\lw(P) = \egon(P)\]
    unless $P\cong d\Sigma$ for some integer $d\geq2$ or $P\cong 2\Upsilon$, in which case $\lw(P)=\egon(P)+1$.
\end{proposition}

\subsection{Metric graphs, gonality, scramble number}

We provide a brief overview of divisor theory on metric graphs. For a more detailed exposition, we refer the reader to \cite{GK08}. A \emph{combinatorial graph} $G$ will refer to a finite multigraph, which may contain self-loops. A \emph{metric graph} $\Gamma$ is a pair $(G, \ell)$ consisting of an underlying combinatorial graph $G$ and a length function $\ell:E(G) \to \R_{> 0}$. This pair $(G, \ell)$ defines a topological space by gluing intervals of the real line according to the structure of $G$. In an abuse of notation, we will also call this space $\Gamma$.

A \emph{divisor} $D$ on $\Gamma$ is a finite formal sum of points of $\Gamma$ with integer coefficients. The \emph{degree} of a divisor $D = \sum_{i} a_i \cdot p_i$ (where $a_i \in \Z$ and $p_i \in \Gamma$) is the sum $\sum_{i}a_i$. A divisor is \emph{effective} if all of its coefficients $a_i$ are non-negative. 

Divisor equivalence is defined using tropical rational functions. A \emph{tropical rational function} is a continuous piecewise-linear function $f: \Gamma \to \R$ with a finite number of integer slopes. Let the \emph{order of $f$ at a point $p \in \Gamma$}, denoted $\operatorname{ord}_p(f)$, be the sum of outgoing slopes of all segments of $\Gamma$ with endpoint at $p$. 

The \emph{divisor associated with $f$} is defined to be
\[(f) := \sum_{P \in \Gamma} \operatorname{ord}_p(f) \cdot p.\]
We say that two divisors $D, D'$ are equivalent if $D - D'$ is the divisor $(f)$ for some tropical rational function $f$. Broadly speaking, one may view this equivalence as a ``continuous'' form of chip-firing.

The \emph{rank} $r(D)$ of a divisor $D$ is $-1$ if $D$ is not equivalent to an effective divisor. Otherwise, $r(D)$ is equal to the maximum integer $r \geq 0$ such that for any effective divisor $E$ of degree $r$, the divisor $D - E$ is equivalent to some effective divisor. The \emph{(divisorial) gonality} of a metric graph $\Gamma$ is defined as the minimum degree of a rank-$1$ divisor.

A metric graph with gonality 2 is called \emph{hyperelliptic}. Several equivalent characterizations of hyperellipticity exist in the literature, such as the existence of a degree-2 harmonic morphism from the graph to a tree, or the existence of an involution whose quotient is a tree \cite[Theorem 1.3]{cha13}.

Now, we briefly introduce the scramble number of a combinatorial graph, an invariant that lower bounds the gonality of a metric graph. For a more detailed exposition, we refer the reader to \cite{echavarria2021scramble}, \cite{new_lower_bound}. 

A \emph{scramble} on a combinatorial graph $G$ is a collection $\cE = \{V_1, \dots, V_\ell\}$ of connected vertex subsets $V_i$, where each $V_i$ is called an \emph{egg}. With every scramble, we associate two parameters: its hitting number $h(\cE)$ and its egg-cut number $e(\cE)$.

A \emph{hitting set} of $\cE$ is a set of vertices $W \subseteq V(G)$ that has nonempty intersection with each egg. Let $h(\cE)$ be the minimum size of a hitting set of $\cE$. An \emph{egg-cut} of $\cE$ is a subset $A \subseteq V(G)$ such that $A$ and $A^C$ each fully contain an egg; the size of the egg-cut is $|E(A, A^C)|$. Let $e(\cE)$ denote the minimum size of an egg-cut. We define the order of a scramble to be
\[\|\cE\| = \min\{h(\cE), e(\cE)\}.\]
Finally, the \emph{scramble number} of a graph $G$, denoted $\sn(G)$, is the maximum order of all scrambles on $G$. In \cite{new_lower_bound} it is shown that the gonality of the finite graph $G$ (defined similarly to the metric case) is lower bounded by its scramble number. In \cite{echavarria2021scramble}, this result is translated to metric graphs. They refer to the canonical loopless model $(G,\ell)$ of a metric graph $\Gamma$, which has no loops and has no vertices of degree $2$ except possibly ones with the same neighbor twice.

\begin{proposition}[{\cite[Lemma 2.15]{echavarria2021scramble}}]\label{prop:snlowerbound}
    Let $(G, \ell)$ be the canonical loopless model of a metric graph $\Gamma$. Then,
\[\sn(G) \leq \gon(\Gamma).\]
\end{proposition}

\subsection{Tropical curves}
We refer to $(\R \cup \{\infty\}, \oplus, \odot)$ as the \emph{min-plus semiring}, where $a \oplus b = \min\{a,b\}$ and $a \odot b = a + b$. A \emph{tropical polynomial} is a polynomial in two variables over the min-plus semiring. In this paper, we consider tropical polynomials in two variables of the form $f(x,y) = \bigoplus_{i,j} a_{i,j} \odot x^i \odot y^j$.

The \emph{tropical curve of $f$} is the set of points in $\R^2$ where the minimum is achieved at least by two monomials. The \emph{Newton polygon} of $f$ is the convex hull of the set $\{(i,j) \in \Z^2: a_{i,j} \neq \infty\}$. Throughout this paper, we will take any tropical curve of $f$ to be \emph{smooth}, which means $(i)$ that for all lattice points $(i,j)$ in the Newton polygon of $f$, $a_{i,j}\neq \infty$ and $(ii)$ that the regular subdivision of the Newton polygon of $f$ induced by the height function $\omega(i,j)=a_{i,j}$ is a unimodular triangulation. There is a beautiful result stating that the tropical curve of $f$ is dual to a subdivision of its Newton polygon. Such a duality is illustrated in the middle of Figure \ref{figure:skeleton}. For more information, see \cite[\S 2.2]{morrison2019tropicalgeometry}.  

    \begin{figure}[hbt]
        \centering
        \includegraphics[width=0.8\linewidth]{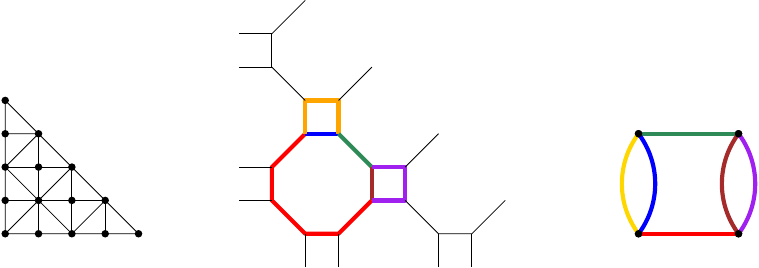}
        \caption{A triangulation of a Newton polygon, a tropical curve, and its skeleton.}
        \label{figure:skeleton}
    \end{figure}

Now, we briefly describe the process wherein a tropical curve deformation retracts to a metric graph. Let $C$ be a tropical curve and assume that the Newton polygon of $C$ has $g \geq 2$. The \emph{skeleton} of $C$ is the metric graph obtained by removing all infinite rays of $C$; retracting any leaves and their attached edges; then ``smoothing'' over resulting 2-valent vertices. The result is a connected, trivalent planar graph of genus $g$.  This is illustrated on the right in Figure \ref{figure:skeleton}.

An equivalence relation is defined on the tropical curves by identifying those that yield isomorphic skeletons. In this paper, we will refer to tropical curves and metric graphs interchangeably—as these skeletons govern the combinatorial structure of tropical curves.

We close this section with an example to tie together the many topics we have discussed.  Consider the genus $5$ polygon $P$ illustrated on the left in Figure \ref{figure:genus_5_cube_example}.  The pictured unimodular triangulation gives rise to a family of tropical curves; one such curve is illustrated in the middle of the figure.  The skeleton of this tropical curve is a metric cube graph, illustrated on the right.

\begin{figure}[hbt]
    \centering
    \includegraphics{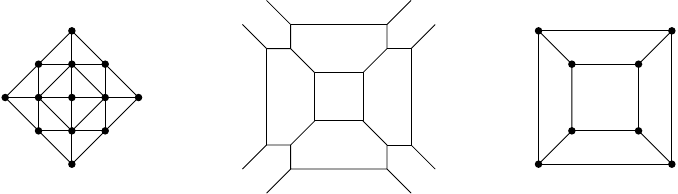}
    \caption{A polygon of genus $5$, a dual tropical curve, and the skeleton $\Gamma$.}
    \label{figure:genus_5_cube_example}
\end{figure}

The interior polygon of $P$ has lattice width $2$, so $\egon(P)=4$.  We will show that the skeleton $\Gamma$ of the tropical curve does indeed have gonality equal to $4$.  Figure \ref{figure:divisor_and_scramble_cube} presents the divisor $D$ on $\Gamma$, with one chip on each of the four topological vertices of $\Gamma$; as well as a scramble on the underlying graph $G$, with four eggs, each consisting of two vertices.  We will argue that this divisor has positive rank, implying $\gon(\Gamma)\leq 4$; and that this scramble has order $4$, implying $4\leq \sn(G)$.  It will follow that $\gon(\Gamma)=4$, as expected.

\begin{figure}[hbt]
    \centering
    \includegraphics{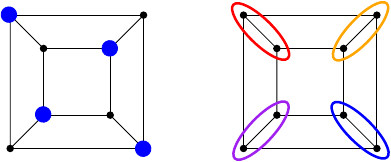}
    \caption{A divisor $D$ of degree $4$, and a scramble consisting of four eggs.}
    \label{figure:divisor_and_scramble_cube}
\end{figure}

To see that the divisor has positive rank, let any $v\in \Gamma$, and let $E=v$ be an effective divisor of degree $1$.  If $v$ is in the support of $D$, then $D-E$ is effective.  Otherwise, $v$ is either a topological vertex $q$ without a chip, or lies on an edge bounded by such a $q$ and a chipped vertex.  Our equivalence relation on divisors allows us to move the three chips neighboring $q$ towards $q$, until one or more reaches it.  If a chip reached $v$ in this process, we are done.  Else, the topological line segment in which $v$ lies has a chip on either side; and these two chips can be moved together until one of them coincides with $v$.  Figure \ref{figure:cube_chip_firing}
demonstrates an example of how to move a chip to $v$, and shows a tropical rational function that exhibits the equivalence of the two divisors.  Thus $r(D)\geq 1$, and $\gon(\Gamma)\leq \deg(D)=4$.

\begin{figure}[hbt]
    \centering
    \includegraphics{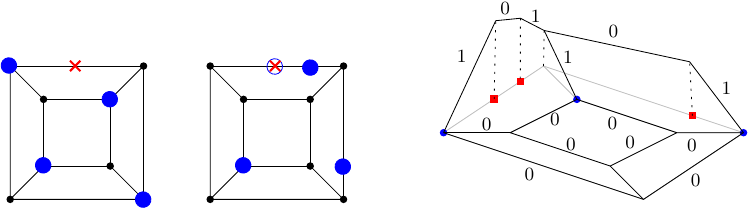}
    \caption{A divisor $D-v$, and an equivalent effective divisor, with a tropical rational function demonstrating equivalence.}
    \label{figure:cube_chip_firing}
\end{figure}

Now consider the scramble $\mathcal{E}$ shown on the right in Figure \ref{figure:divisor_and_scramble_cube}\footnote{This scramble is also utilized in \cite{platonic} to show that the cube has scramble number equal to $4$.  We present a self-contained argument here for completeness.}.  Since the four eggs are disjoint, we have hitting number $h(\mathcal{E})=4$.  For the egg-cut number, we note that between any pair of eggs, there exist four pairwise edge-disjoint paths: two going around the outer cycle, and two going around the inner cycle.  Any egg-cut separating a particular pair of eggs must include an edge from each path, so $e(\mathcal{E})\geq 4$.  It follows that the scramble order $||\mathcal{E}||=4$, meaning $\sn(G)\geq 4$.  Combined with $\gon(\Gamma)\leq 4$, we have $4\leq \sn(G)\leq \gon(\Gamma)\leq 4$, so $\gon(\Gamma)=4$.

\section{Algebraic and tropical moduli spaces}\label{sec:AG}

In this section we describe moduli spaces both of algebraic curves and of tropical curves.  Although our focus is on the tropical world, much of the inspiration and several important results come from the algebraic setting.

Let $g\geq 2$.  Denote by $\mathcal{M}_g$ the moduli space of algebraic curves of genus $g$; this is a $(3g-3)$-dimensional space.  A smaller space of genus $g$ curves was introduced by \cite{castryck_voight}, who considered non-degenerate\footnote{Here non-degenerate means the polynomial has no singularities in the torus $(k^*)^2$, nor do the polynomials obtained by restricting to the terms coming from a face of the Newton polygon.} curves defined by Laurent polynomials in two variables.  The moduli space of such curves of genus $g$ is denoted $\mathcal{M}_g^\textrm{nd}$, and can be decomposed as a finite union of smaller spaces, stratified by the Newton polygon of the curve:
\[\mathcal{M}_g^\textrm{nd}=\bigcup_P \mathcal{M}_P.\]
Here the union is over all equivalence classes of lattice polygons $P$ with $g$ interior lattice points, and $\mathcal{M}_P$ is the set of all non-degenerate curves with Newton polygon $P$.

Zooming in on a particular non-degenerate curve, we find that its gonality (that is, the minimum degree of a positive rank divisor on that curve) can be computed easily from the Newton polygon.

\begin{theorem}[{\cite[Corollary 6.2]{CC17}}] \label{theorem:gon=egon_algebraic} Let $f\in k[x^\pm,y^\pm]$ be irreducible and non-degenerate with respect to its Newton polygon $P$, and let $U(f)$ be the curve on the 2-dimensional torus defined by $f$. Then the gonality $U(f)$ equals $\egon(P)$.
\end{theorem}

The tropical analog of $\mathcal{M}_g$, denoted $\mathbb{M}_g$, encodes all metric graphs of genus $g$.  It has the structure of a stacky fan, with top-dimensional cells corresponding to trivalent graphs of genus $g$ (with loops and multi-edges allowed).  These cells are glued according to a poset structure determined by which graphs become identical when edge lengths shrink to $0$; see \cite{tropical_torelli} and \cite{MR2968636} for more details.  Much like $\mathcal{M}_g$, this space is $(3g-3)$-dimensional, as can be seen by counting the edges of a trivalent graph.

The authors of \cite{brodsky_joswig_morrison_sturmfels} introduced a tropical analog of $\mathcal{M}_g^\textrm{nd}$ as a subset of $\mathbb{M}_g$, consisting of all metric graphs that arise as the skeleton of a smooth plane tropical curve of genus $g$.  We denote the space of such metric graphs as $\mathbb{M}_g^\textrm{nd}$. (The original paper \cite{brodsky_joswig_morrison_sturmfels} used $\mathbb{M}_g^\textrm{planar}$ to denote this space; we instead use $\mathbb{M}_g^\textrm{nd}$ to emphasize the relationship with $\mathcal{M}_g^\textrm{nd}$.)

As in the algebraic case, $\mathbb{M}_g^\textrm{nd}$ admits a decomposition into smaller spaces.  First we may write
\[\mathbb{M}_g^\textrm{nd} = \bigcup_P \mathbb{M}_P,\]
where $\mathbb{M}_P$ denotes the set of all metric graphs arising from tropical curves with Newton polygon $P$. This union is over all lattice polygons $P$ with $g$ interior points; since two equivalent lattice polygons give rise to the same metric graphs, we may choose one polygon from each equivalence class, making the union finite. We can be even more economical by restricting to lattice polygons that are maximal with respect to inclusion among polygons with $g$ interior lattice points; see \cite[Lemma 2.6]{brodsky_joswig_morrison_sturmfels}. 

We can decompose $\mathbb{M}_P$ further, into a finite union over all regular unimodular triangulations $\Delta$ of $P$:
\[\mathbb{M}_P=\bigcup_{\Delta}\mathbb{M}_\Delta.\]
Here $\mathbb{M}_\Delta$ denotes the moduli space of all metric graphs arising from smooth tropical curves dual to the particular triangulation $\Delta$ of $P$. Section 3 in \cite{brodsky_joswig_morrison_sturmfels} interprets $\mathbb{M}_\Delta$ as the image of the secondary cone of $\Delta$ under a certain linear map.

We now define two loci within $\mathbb{M}^{\textrm{nd}}_{g}$ based on expected gonality and gonality. First, we consider the locus of curves whose expected gonality is $d$. Let $\mndexp{g}{d}$ be the locus within $\mathbb{M}_g^\textrm{nd}$ of metric graphs with genus $g$ that are the skeleton of some smooth plane tropical curve whose Newton polygon has expected gonality $d$: 
\[\mathbb{M}_{g,\underline{d}}^{\textrm{nd}}=\!\!\!\!\!\!\!\!\bigcup_{\stackrel{g(P)=g}{\textrm{egon}(P)=d}} \!\!\!\!\!\!\!\mathbb{M}_P.\]

Since this space is a finite union of polyhedral cones, we define its dimension to be the maximum dimension of its constituent cones:
\[\dmndexp{g}{d} = \!\!\!\!\max_{\stackrel{g(P)=g}{\textrm{egon}(P)=d}} \!\!\!\dim\left(\modspace{P}\right).\]
Alternatively, we can consider the locus of genus $g$ graphs in $\mathbb{M}_g^{\textrm{nd}}$ with a prescribed gonality $d$, denoted $\mnd{g}{d}$:
\[\mathbb{M}_{g,d}^{\textrm{nd}}=\{\Gamma\in \mathbb{M}_g^{\textrm{nd}}\,|\, \gon(\Gamma)=d\}.\]
The geometric structure of $\mnd{g}{d}$ in generality is unknown. In particular, it is unknown if it also admits a decomposition into a (finite) union of polyhedral cones. Despite this difficulty, we may still consider $\dim\bigl(\mnd{g}{d}\bigr)$, for instance, as the maximum dimension of a ball contained in $\mnd{g}{d}$.

Next, we state some straightforward results relating $\gon(\Gamma)$ and $\textrm{egon}(P)$ as a consequence of the literature.

\begin{lemma}\label{lma:gon_at_most_egon_graphs} Let $\Gamma$ be the skeleton of a smooth plane tropical curve with Newton polygon $P$. Then $\gon(\Gamma)\leq \egon(P)$.
\end{lemma}

\begin{proof} The case where the metric graph has genus $0$ (forcing gonality $1$ and expected gonality $1$) or genus $1$ (forcing gonality $2$ and expected gonality $2$) are handled immediately.  For $g\geq 2$, we
   let $F$ denote the tropical polynomial defining the tropical curve, and let $f\in k[x^\pm,y^\pm]$ be a polynomial tropicalizing to $F$.  Since the tropical curve is smooth, we know that $U(f)$ is non-degenerate.  By Theorem \ref{theorem:gon=egon_algebraic}, $U(f)$ has gonality $\egon(P)$.  By \cite[Corollary 3.2]{baker}, the metric graph $\Gamma'$ associated to $U(f)$ has gonality at most $\egon(P)$.  By \cite[Corollary 5.28]{bpr}, $\Gamma$ is a faithful representation of $\Gamma'$, and so has the same gonality.  Thus $\gon(\Gamma)\leq \egon(P)$.
\end{proof}

For sufficiently small genus and expected gonality, we can strengthen this result to equality. Theorem~\ref{thm:gon=egon_small_egon} will show the case where $\egon(P) \leq 3$. In Appendix~\ref{sec: tropical_Appendix}, we show the equality for $g \leq 4$ (including $g=0, 1$). We do not know of any smooth plane tropical curve with Newton polygon $P$ and skeleton $\Gamma$ such that $\gon(\Gamma)<\egon(P)$; indeed, the content of Conjecture \ref{conj:gon=egon} is that no such tropical curve exists.

Before proceeding, we present a slight reformulation of a result from \cite{tropical_hyperelliptic_curves_in_the_plane}:

\begin{theorem}\label{thm:ralph_hyperelliptic}  Let $g\geq 1$, and suppose $P$ is a polygon of genus $g$ and $\Gamma$ is the skeleton of any smooth tropical curve with Newton polygon $P$. Then, $P$ is hyperelliptic if and only if $\Gamma$ is hyperelliptic (that is, $\gon(\Gamma) = 2$).
\end{theorem}

\begin{proof}
    The reverse implication is given by \cite[Theorem 1.1]{tropical_hyperelliptic_curves_in_the_plane}. Suppose $P$ is a hyperelliptic polygon; this implies that $\egon(P) = 2$. By Lemma~\ref{lma:gon_at_most_egon_graphs}, one has $\gon(\Gamma) \leq 2$. Since $g \geq 1$, $\Gamma$ has a cycle and thus is not a tree. A quick corollary of the tropical Riemann-Roch theorem~\cite{bn, GK08} is that a metric graph has gonality $1$ if and only if it is a tree. So $\gon(\Gamma) = 2$. 
    \end{proof}
    

This allows us to obtain an equality when our polygon has low expected gonality.

\begin{theorem}\label{thm:gon=egon_small_egon}
    Let $\Gamma$ be the skeleton of a smooth plane tropical curve with Newton polygon $P$ with $\egon(P) \leq 3$. Then,
    \[\gon(\Gamma) = \egon(P).\]
\end{theorem}

\begin{proof}
    Again, we use that the tropical Riemann-Roch theorem~\cite{bn, GK08} implies that a metric graph has gonality $1$ if and only if it is a tree. We stratify by expected gonality of $P$:
    \begin{itemize}
        \item \underline{Case 1: $\egon(P) = 1$.}  Let $\Gamma$ be the skeleton of a tropical curve arising from a polygon $P$ with $\egon(P) = 1$. Recall that $\egon(P)$ is defined to be $1$ when $P^{(1)}$ is empty. Therefore, $g(P) = 0$, and $\gon(\Gamma)  = 1$. 
        \item \underline{Case 2: $\egon(P) = 2$.} Let $\Gamma$ be the skeleton of a tropical curve arising from a polygon $P$ with $\egon(P) = 2$. By Theorem~\ref{thm:ralph_hyperelliptic}, $\gon(\Gamma) = 2$.
        \item \underline{Case 3: $\egon(P) = 3$.} Let $\Gamma$ be the skeleton of a tropical curve arising from a polygon $P$ with $\egon(P) = 3$. Then, $P$ has a non-degenerate interior. By Theorem~\ref{thm:ralph_hyperelliptic}, $\gon(\Gamma) \neq 2$. Because genus is positive, $\gon(\Gamma) \neq 1$ as well. By Lemma~\ref{lma:gon_at_most_egon_graphs}, $\gon(\Gamma) \leq \egon(P) =  3$. Thus, $\gon(\Gamma) = 3$.
    \end{itemize}
    \end{proof}

We may also slightly reformulate the above theorems in the language of moduli spaces:

\begin{theorem} \label{thm:d_is_2}
The equality $\mnd{g}{2} = \mndexp{g}{2}$ holds.
\end{theorem}

\begin{proof}
    This follows directly from Theorem~\ref{thm:ralph_hyperelliptic}.
\end{proof}

\begin{theorem} \label{thm:d_is_3}
The inclusion $\mndexp{g}{3} \subseteq \mnd{g}{3}$ holds.
\end{theorem}
\begin{proof}
    This follows from the $\egon(P) = 3$ case of Theorem~\ref{thm:gon=egon_small_egon}.
\end{proof}

We conclude by discussing some prior work on the dimension of the tropical moduli spaces outlined above. In \cite{fixed_newton_polyon}, the authors develop the following formula for $\dim(\modspace{P})$:
\begin{lemma}[{\cite[Theorem 1.4]{fixed_newton_polyon}}]\label{lma:column_vector_formula}
    Let $P$ be a maximal non-hyperelliptic polygon with $g(P)$ interior lattice points, $r(P)$ boundary lattice points, and $c(P)$ column vectors. Then,
    \[\dim(\modspace{P}) = g(P) + r(P)  - 3 - c(P).\]
\end{lemma}

The authors of \cite{fixed_newton_polyon} prove the above theorem by identifying a class of triangulations of $P$ called beehive triangulations; we define these in the next section. These triangulations are significant because they are top-dimensional within $\mathbb{M}_{P}$. More specifically,
\begin{lemma}[{\cite[Lemma 3.7]{fixed_newton_polyon}}]\label{lma:beehive}
    Let $\Delta$ be a regular beehive triangulation of a maximal non-hyperelliptic polygon $P$. Then, 
    \[\dim(\modspace{\Delta}) = \dim(\modspace{P}).\]
\end{lemma}

\section{A lower bound on $\dmnd{g}{d}$ for large genus}\label{sec:lower}

    In this section, we build towards a proof of Theorem~\ref{tmd:equaldim} by proving:

    \begin{proposition}
        \label{prop:equaldimgreater}
        Fix $d \geq 3$. For $g \geq d^3$,  
        \[\dim\Bigl(\mnd{g}{d}\Bigr) \geq \dmndexp{g}{d}.\]
    \end{proposition}
    
    We proceed by showing that when the genus of a maximal non-hyperelliptic polygon $P$ is sufficiently large compared to its lattice width, then $P$ admits a triangulation $\Delta$ such that the skeleton $\Gamma$ of any tropical curve  dual to $\Delta$ has gonality equal to the expected gonality of $P$. Furthermore, we prove that $\modspace{\Delta}$ is top-dimensional in $\mathbb{M}_P$. Throughout, the central element of each argument is a geometric shape called a \emph{crystal}. 
    
    This section is split into four parts. In the first, we define beehive triangulations and set some notation. In the second, for polygons $P$ containing a crystal, we construct a regular triangulation $\Delta$ such that $\modspace{\Delta}$ is top-dimensional in $\mndexp{g}{d}$ where $g=g(P)$ and $d=\egon(P)$. In the third, we prove that any metric graph dual to $\Delta$ has gonality equal to the expected gonality of $P$. Finally, we prove that if $g\geq d^3$, then $P$ must contain a crystal.

    \subsection{Beehive triangulations and notation}

    Throughout this section, unless otherwise noted, assume that any lattice polygon $P$ is maximal with $\lw(P) = \egon(P) = d$\footnote{Recall Proposition~\ref{lwisegon}, which states that $\lw(P) = \egon(P)$ for all polygons \emph{except}  $2\Upsilon$ and the family $q\Sigma$ for $q \geq 2$. We may check that none of these polygons satisfies $g \geq \egon(P)^3$, so we may safely exclude them from consideration.} and $d \geq 3$. Additionally, assume that $P$ lies in the horizontal strip $\R \times [0, d]$.

    Label the vertices of $P^{(1)}$ by $v_1, \dots, v_n$ counterclockwise. Define $\tau_i$ to be the face $\conv\{v_{i-1},v_i\}$, with indices taken modulo $n$. Recall that $\tau_i^{(-1)}$ is the ``relaxed'' face corresponding to $\tau_i$, and $\mathcal{H}^{(-1)}_{\tau}$ is its corresponding half-space (see the definitions in Section~\ref{sec:bg-lattice_polygons}). Define $v_i^{(-1)} = \partial \cH^{(-1)}_{\tau_{i}}\cap\partial\cH^{(-1)}_{\tau_{i+1}}.$ Observe that here, the $v_i^{(-1)}$'s are vertices of $P$ as a consequence of $P$ being maximal, as $v_i^{(-1)}$ is a vertex of $(P^{(1)})^{(-1)} = P$ by Proposition~\ref{prop:pushinpushout}. 
    
    We say a unimodular triangulation $\Delta$ of $P$ is a \emph{beehive triangulation} of $P$ if
    \begin{enumerate}
        \item $\Delta$ includes all boundary edges of $P^{(1)}$;
        \item $v_i$ is connected to $v_i^{(-1)}$ for all $i$;
        \item for each $i$, the number of lattice points on $\tau_i$ connected to at least two lattice points on $\tau^{(-1)}_i$ is maximized over all unimodular triangulations of $P$.
    \end{enumerate}
    In this paper, we will call the points that satisfy the third criterion \emph{good} points on $\tau_i$\footnote{A single lattice point can be a good point on two different edges.}. Some examples of beehive triangulations are shown in Figure~\ref{fig:beehives}.

    \begin{figure}[hbt]
        \centering
        \includegraphics[width=0.8\linewidth]{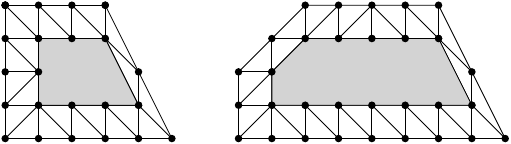}
        \caption{Two examples of beehive triangulations, where the shaded region can be any unimodular triangulation.}
        \label{fig:beehives}
    \end{figure}

    \subsection{Construction of $\Delta$}\label{sec:delta}

    We say that a polygon $P$ \emph{contains a crystal of length $d$} if there exist $d+1$ consecutive integers $x_0,\dots,x_d$, such that for every $i \in \{0,\dots,d\}$ and every $j \in \{1,\dots,d-1\}$, the point $(x_i,j)$ is an interior point of $P$. In other words, $P$ contains a crystal of length $d$ if and only if it contains a rectangle of interior lattice points with width $d$ and height $d-2$. See Figure~\ref{figure:crystal} for an example of a crystal.

    \begin{figure}[hbt]
        \centering
        \includegraphics[width=0.8\linewidth]{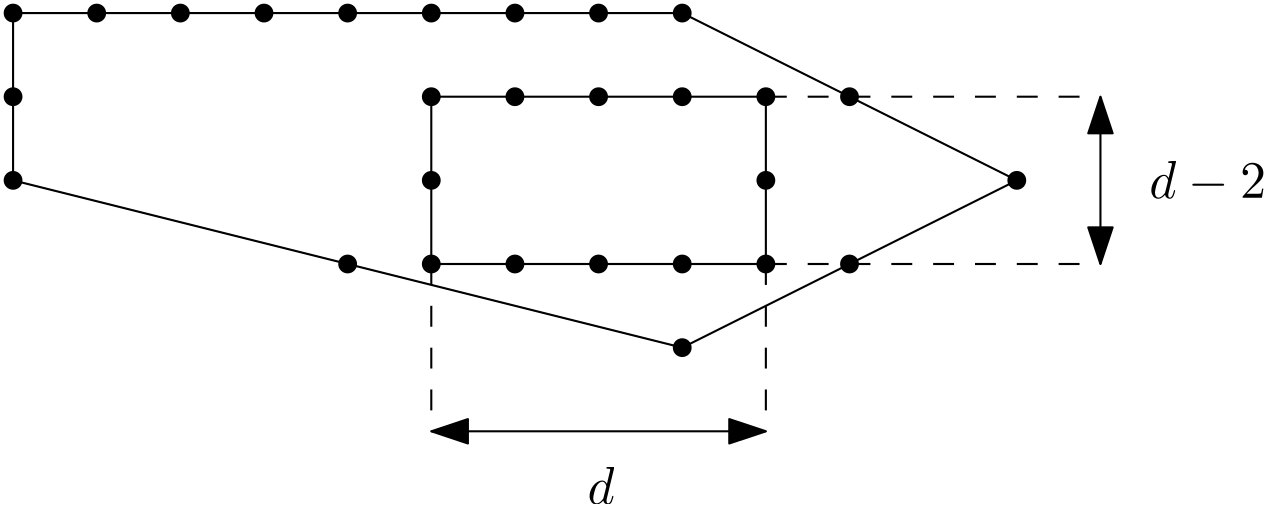}
        \caption{A polygon with a crystal of length $d = 4$.}
        \label{figure:crystal}
    \end{figure}

    Let $P$ be a lattice polygon containing a crystal of length $d$. Our main goal in this section is to construct a beehive triangulation $\Delta$ that respects the crystal structure. We state two lemmas that are useful in constructing $\Delta$.

    \begin{lemma}[{\cite[Proposition 2.3.16]{Triangulations}}]\label{lma:heightrefinement}
        Let $\cS$ be a regular subdivision of a polygon $P$ and let $h:P \cap \Z^2 \to \R$ be a height function on the lattice points of $P$. Then the following is a regular subdivision that refines $\cS$:
        \[\cS_h := \bigcup_{C \in \cS} \cR\left(P\vert_C, h\right).\]
        We say that $\cS_h$ is the subdivision $\cS$ \emph{refined by $h$}.
    \end{lemma}

    \begin{lemma}[{\cite[Lemma 3.3]{counting_lattice_triangulations}}]\label{lma:patching}
    Let $\Delta$ be a unimodular triangulation of a lattice trapezoid with two parallel sides $\tau$ and $\tau'$ of distance one. Then, every piecewise linear function $h: \tau \to \R$ that is strictly convex on $\tau \cap \Z^2$ can
    be extended to a height function for $\Delta$.
    \end{lemma}

    An immediate consequence of Lemma \ref{lma:patching} is that any subdivision of a polygon with lattice width $1$ is regular.
    
    Consider the initial subdivision $\cS$ of $P$, induced by the height function
    \[h\left(v\right) = \begin{cases}
        1 & v \in \partial P\cap \Z^2 \\
        0 & \text{elsewhere on $P \cap \Z^2$.}
    \end{cases}\]
    An example of this subdivision is illustrated in Figure \ref{figure:initial_subdivision.pdf}.

        \begin{figure}[hbt]
        \centering
        \includegraphics[width=0.8\linewidth]{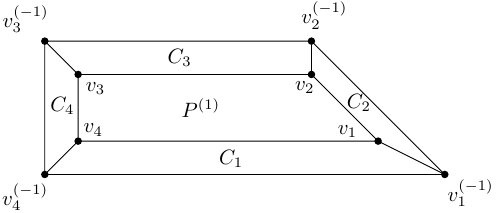}
        \caption{Initial subdivision $\cS$, with vertices labeled.}
        \label{figure:initial_subdivision.pdf}
        \end{figure}
    Define cells $C_i = \conv\{\tau_i, \tau_i^{(-1)}\}$. Observe that the cells of $\cS$ are $C_1, \dots, C_n$, as well as $P^{(1)}$.  
    
    Let
    \[h'(v) = \begin{cases}
        1 & v \in \{v_{1}^{(-1)}, \dots, v_n^{(-1)}\} \\
        0 & \text{elsewhere on $P \cap \Z^2$},
    \end{cases}\]
    and let $\cS'$ denote the refinement of $h'$ on $\cS$ (as defined in Lemma~\ref{lma:heightrefinement}). See Figure \ref{figure:initial_subdivision2.pdf}.  We describe this refinement explicitly below.
        \begin{figure}[hbt]
        \centering
        \includegraphics[width=0.8\linewidth]{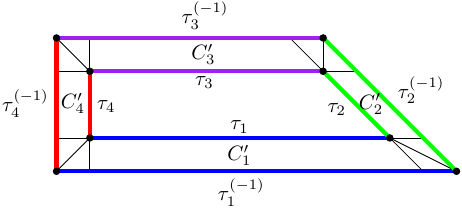}
        \caption{The subdivision $\cS'$ obtained by refining $\cS$, with edges labeled.}
        \label{figure:initial_subdivision2.pdf}
        \end{figure}

    If $\tau_i^{(-1)}$ has more than two distinct lattice points, then on each $C_i$, refining by $h'$ will connect $v_i$ to the lattice point on $\tau_i^{(-1)}$ that is closest to and distinct from $v_i$. Similarly, $v_{i-1}$ is connected to the lattice point on $\tau_i^{(-1)}$ that is closest to and distinct from $v_{i-1}$. This creates two triangles of area $1/2$. Take $C_i'$ to be the cell between these two triangles of area $1/2$. If $\tau_i^{(-1)}$ has two or fewer lattice points, then refining by $h'$ will not create any additional edges on $C_i$. In this case, take $C_i' = C_i$.

    We will now construct a height function on $P^{(1)}$. Let $x_L$ and $x_R$ be the minimum and maximum $x$-coordinates in $P^{(1)}$ respectively. Let $\omega: P \cap \Z^2 \to \R$ by
    \[\omega(x,y) = (x-x_L)(x-x_R) + (y-1)(y-(d-1)).\]

    Since $\omega$ is a sum of two separable quadratic functions with positive leading coefficients, $\omega$ is strictly convex. We focus on the lower convex hull of the graph
    \[\Lambda = \{(x,y,\omega(x,y)): (x,y) \in P \cap \Z^2\},\]
    as well as its projection onto $P^{(1)}$, denoted $\cS_\omega$.
    
    A key feature of $\omega$ is its behavior on the grid. Consider any unit square whose four lattice-point corners $(x,y), (x+1, y), (x, y+1), (x+1, y+1)$ are all in $P^{(1)}$. 
    Then, the four corresponding points in $\Lambda$ are coplanar (as a consequence of the separability of $\omega$), and the image of any lattice point outside of the unit square will lie strictly above the plane. Therefore, the four points form a face on the lower convex hull of $\Lambda$. Projecting these faces back onto $P^{(1)}$ reveals a grid-aligned subdivision of $P^{(1)}$, whose edges connect all points along the same $x$-coordinates, and similar for the $y$-coordinates.
    

    The main challenge from this point is showing that there exists a \emph{beehive} triangulation $\Delta$ of $P$ that refines the ``grid-like'' subdivision that $\omega$ induces on $P^{(1)}$. We prove two technical lemmas before showing that such a triangulation $\Delta$ does exist.

    \begin{proposition}\label{prop:threepointbeehive}
        Let $C$ be a lattice polygon such that $C = \conv\{\tau, \tau'\}$, where $\tau, \tau'$ are parallel edges separated by distance 1. Suppose $\tau'$ has three or more lattice points. Label the lattice points of $\tau$ as $\nu_1, \dots, \nu_n$ and the lattice points of $\tau'$ as $\mu_1, \dots, \mu_m$. Then, there exists a regular unimodular triangulation $\Delta$ of $C$ such that:
        \begin{itemize}
            \item $\nu_1$ is connected to $\mu_2$;
            \item $\nu_n$ is connected to $\mu_{m-1}$;
            \item the number of \emph{good} points on $\tau$ (i.e., those connected to at least two points on $\tau'$) is maximized over all unimodular triangulations of $C$.
        \end{itemize}
    \end{proposition}

    \begin{proof}
        In \cite[Proposition 4.6]{fixed_newton_polyon}, the authors show that the maximum number of good points on $\tau$ is $\min\{n, m-1\}$. This can be achieved by a ``zig-zag'' unimodular triangulation that connects, in sequence, $\nu_1, \mu_1, \nu_2, \mu_2, \dots$ and terminates when there are no points left on either $\tau$ or $\tau'$.

        First, suppose $\tau$ contains only two lattice points. Then, it can be seen that connecting $\nu_1, \mu_2$ and $\nu_n, \mu_{m-1}$ will connect all points on $\tau$ to at least two points on $\tau'$. Now, suppose that $\tau$ contains at least three lattice points. Connect $\nu_1, \mu_2, \nu_2$ and $\nu_n, \mu_{m-1}, \nu_{n-1}$, and call $C'$ the remaining space to triangulate. By \cite[Proposition 4.6]{fixed_newton_polyon}, $C'$ can be triangulated such that there are $\min\{n-2, m-3\}$ good points on $\tau$. Use this triangulation to finish the unimodular triangulation $\Delta$ on $C$, which will have
        \[2 + \min\{n-2, m-3\} = \min\{n, m-1\}\]
        good points on $\tau$. Lastly, $\Delta$ is regular by Lemma \ref{lma:patching} since $C$ is lattice width 1.
    \end{proof}

    \begin{proposition}[{\cite[Proposition 4.5]{fixed_newton_polyon}}]\label{prop:twopointbeehive}
        Let $C$ be a lattice polygon such that $C = \conv(\tau, \tau')$, where $\tau, \tau'$ are parallel edges separated by distance 1. Suppose $\tau'$ has two or fewer lattice points. Then, any unimodular triangulation $\Delta$ of $C$ will maximize the number of good points on $\tau$.
    \end{proposition}

    \begin{proof}      
        Suppose $\tau'$ has only one lattice point. Then, there is only one possible unimodular triangulation of $C$, namely the one that connects $\tau'$ to every point on $\tau$. Suppose $\tau'$ has two lattice points. Then, $\Delta$ must connect two points on $\tau'$ to the same point on $\tau$, which is maximal. 
    \end{proof}

    \begin{lemma}\label{lma:existsbeehive}
        There exists a regular beehive triangulation $\Delta$ of $P$ such that $\Delta$ refines $\cS_{\omega}$ on $P^{(1)}$.
    \end{lemma}

    \begin{proof}    
        Notice that $\cS'$ already contains all the boundary edges of $P^{(1)}$; furthermore, $\cS'$ connects each $v_i$ to $v_i^{(-1)}$. Therefore, any refinement of $\cS'$ must, as well. In addition to the two criteria above, a beehive triangulation must maximize, for each cell $C_i$, the number of good points on $\tau_{i}$. Propositions \ref{prop:threepointbeehive} and \ref{prop:twopointbeehive} guarantee such refinements on $C_i$ that are ``compatible'' with $\cS_\omega$, and it remains to ``patch'' these refinements onto $\cS_\omega$.
    
        We proceed by constructing a height function $g: P \cap \Z^2 \to \R$ that refines $\cS'$.

        \begin{enumerate}
            \item \textbf{The interior $P^{(1)}$:} For lattice points $(x,y) \in P^{(1)}$, assign them the height $\omega(x,y)$. This choice ensures that the subdivision induced by $g$ on $P^{(1)}$ is $\cS_\omega$ as described earlier. 
        
            \item \textbf{Boundary regions $C'_i$:} Each cell $C_i = \operatorname{conv}\{\tau_i, \tau_i^{(-1)}\}$ is a trapezoid of lattice width 1.  The heights of vertices on $\tau_i$ are already fixed by step 1.

            Let $\tau^{(-1)}_i$ have three or more lattice points on it. Proposition~\ref{prop:threepointbeehive} supplies a triangulation that maximizes good points on $\tau_i$. We observe that it is sufficient to prescribe heights for points on $C_i'$\footnote{Prescribing points on the entirety of $C_i$ may cause conflicts at the vertices of $P$ between two cells.} to induce such a triangulation.
                         
            Notice that because the restriction of a strictly convex function to any line is also strictly convex \cite[\textsection 3.1.1]{BV04}, $\omega$ is strictly convex on $\tau_i$. We may use Lemma \ref{lma:patching} to pick heights on $C_i'$ that induce this triangulation on $C_i'$.
        
            Now suppose $\tau_i^{(-1)}$ has two or fewer points. Then, assign the points in $C_i$ any arbitrary height. By Proposition \ref{prop:twopointbeehive}, any unimodular triangulation suffices for the number of good points on $\tau_i$ to be maximized. 
               
            \item \textbf{Remaining vertices:} On the vertices of $P$ that are not already assigned some height, we also pick any arbitrary height. These heights do not induce any further subdivision.
        \end{enumerate}
        
        Let $\widetilde{\Delta}$ be the subdivision obtained by refining $\cS'$ with $g$. Notice that $\widetilde{\Delta}$ need not be a unimodular triangulation. However, we repeatedly use Lemma \ref{lma:heightrefinement} to further refine $\widetilde{\Delta}$ until we reach a unimodular triangulation. Note that this process also preserves regularity. Thus, we let $\Delta$ be a unimodular triangulation that refines $\widetilde{\Delta}$ via this process; such a triangulation is beehive as it maximizes the number of good points on each $\tau_i$.
    \end{proof}

\subsection{Curves of gonality $d$}

    We now show that any curve that arises from $\Delta$ has gonality equal to $\lw(P) = d$. 

    \begin{lemma}\label{lma:crystalsn}
        Suppose $P$ contains a crystal of length $d$, and let $\Delta$ be a unimodular triangulation that refines $\widetilde{\Delta}$, as described above in Section~\ref{sec:delta}. Let $\Gamma$ be any tropical curve arising from $\Delta$. Then, $\gon(\Gamma) = \lw(P) = d$.
    \end{lemma}
    
    \begin{proof}
        Let $G$ be the dual graph to $\Delta$. Equivalently, $G$ is the underlying combinatorial graph of the skeleton $\Gamma$ \emph{before} smoothing over $2$-valent vertices. Label the integer $x$-coordinates of the crystal $x_0, \dots, x_d$ in increasing order. Consider $\Delta$ restricted to the vertical strip $\mathcal{V} = [x_0, x_d] \times \R$. Denote this subdivision $\Delta\big\vert_{\mathcal{V}}$. Let $G'$ be the dual graph to $\Delta\big\vert_{\mathcal{V}}$.

        Each triangle in $\Delta\big\vert_{\mathcal{V}}$ corresponds to a vertex in $G'$. Let the egg $\cE_i$ be the set of vertices corresponding to triangles strictly lying in the rectangular strip $[x_{i-1}, x_i] \times [1, d-1]$, as well as the triangles immediately above and below the strip. See Figure~\ref{figure:crystal_scramble} for an example.
        
        \begin{figure}[hbt]
            \centering
            \includegraphics[width=0.9\linewidth]{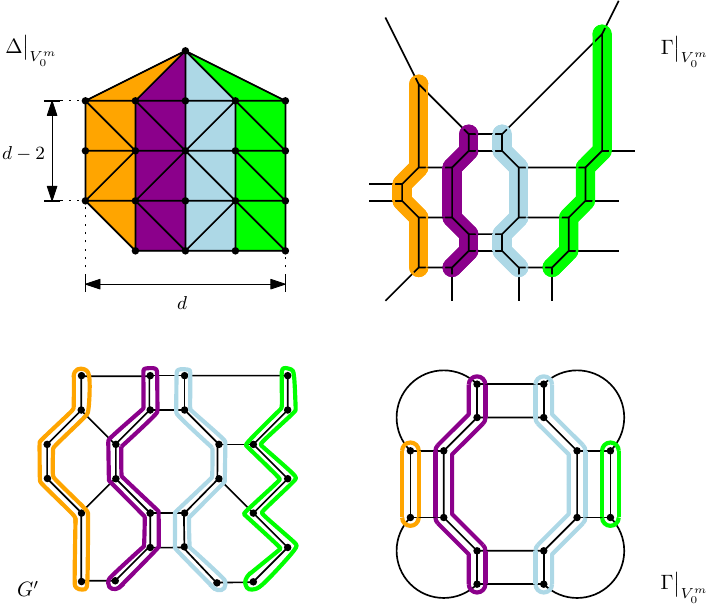}
            \caption{On the left, the subdivision $\Delta\big\vert_{\mathcal{V}}$ and the combinatorial graph $G'$. On the right is the substructure of $\Gamma$ corresponding to $\Delta\big\vert_{\mathcal{V}}$ as a tropical curve and as a metric graph. Each has features corresponding to eggs $\cE_1, \dots, \cE_4$ colored.}
            \label{figure:crystal_scramble}
        \end{figure}

        We exhibit $\{\cE_1, \dots, \cE_d\}$ as a scramble of order $d$. It has hitting number $d$ because there are $d$ disjoint eggs. We claim it also has egg-cut number $d$. Because edges of $G'$ correspond to edges shared between two faces in $\Delta\big\vert_{\mathcal{V}}$, paths that separate $\Delta\big\vert_{\mathcal{V}}$ into two connected nonempty regions correspond to edge-cuts of $G'$.

        Recall that an egg-cut is an edge-cut that contains two complete eggs. If a path on $\Delta\big\vert_{\mathcal{V}}$ corresponds to an egg-cut of $G'$, then the path must have endpoints on the top and bottom sides of $\Delta\big\vert_{\mathcal{V}}$. Furthermore, all edges of $\Delta\big\vert_{\mathcal{V}}$, when decomposed as a vector, have $y$-components that are $0$ or $1$. Thus, all edge paths from the top to the bottom of $\Delta\big\vert_{\mathcal{V}}$ must be of length at least $d$, so an egg-cut must have at least $d$ edges. Therefore, $\sn(G') \geq d$.
        
        Since scramble number is subgraph monotone~\cite[Proposition 4.5]{new_lower_bound}, $\sn(G) \geq  \sn(G') \geq d$. Furthermore, scramble number is subdivision monotone~\cite[Proposition 4.6]{new_lower_bound}, so we can consider $G$ and the combinatorial graph underlying the skeleton of $\Gamma$ to have the same scramble number. We also have, by Proposition~\ref{prop:snlowerbound}, that $\gon(\Gamma) \geq \sn(G) \geq d$. Finally, by Lemma~\ref{lma:gon_at_most_egon_graphs}, $\gon(\Gamma) \leq \egon(P) = \lw(P) = d$.
    \end{proof}

    \subsection{On crystals}
        
    In this section, we show that a crystal must exist if $g \geq d^3$. Let $\tau_1, \tau_k$ be horizontal faces of $P^{(1)}$ lying on $y = 1$ and $y = d-1$ respectively. While it is possible they may be degenerate (or non-existent) for an arbitrary polygon, these edges are non-degenerate in all of the cases we consider below.
    
    \begin{proposition}\label{prop:crystalcondition}
        Let $P$ be a polygon of lattice width $d \geq 3$ with at least $2d - 2$ points on $\tau_1$ and $2d - 2$ points on $\tau_k$. Then, $P$ is equivalent (under a $\Z$-affine transformation) to a polygon $P'$ that contains a crystal of length $d$.
    \end{proposition}

    \begin{proof}
        In the case that there does not already exist such a crystal in $P$, we argue that there exists a shear transformation described by, for some $\ell \in \Z$,
        \[\begin{pmatrix}
            x \\ y
        \end{pmatrix}\mapsto \begin{pmatrix}
            1 & \ell \\
            0 & 1
        \end{pmatrix} \begin{pmatrix}
            x \\ y
        \end{pmatrix}\]
        that transforms $P$ into a polygon $P'$ that ``better aligns'' the $2d - 2$ points on $\tau_1$ and $\tau_k$. 
        
        Observe that a transformation with $\abs{\ell} = 1$ shifts the relative $x$-coordinates of the two edges by $d - 2$. Therefore, by the pigeonhole principle, there exists some shear such that only at most $\floor{\frac{d - 2}{2}}$ points on $\tau_1$ are not aligned with points on $\tau_k$. Then, the number of aligned points must satisfy
        \[(2d - 2) -  \floor{\frac{d - 2}{2}} \geq d + 1.\]
        Note that the last inequality follows after considering only $d \geq 3$. Thus, we guarantee there is a crystal of length $d$ in some polygon equivalent to $P$ (as such a crystal requires $d+1$ aligned points). See Figure~\ref{figure:shear} for an example of such a shear.
        \end{proof}

        \begin{figure}[hbt]
            \centering
            \includegraphics[width=0.9\linewidth]{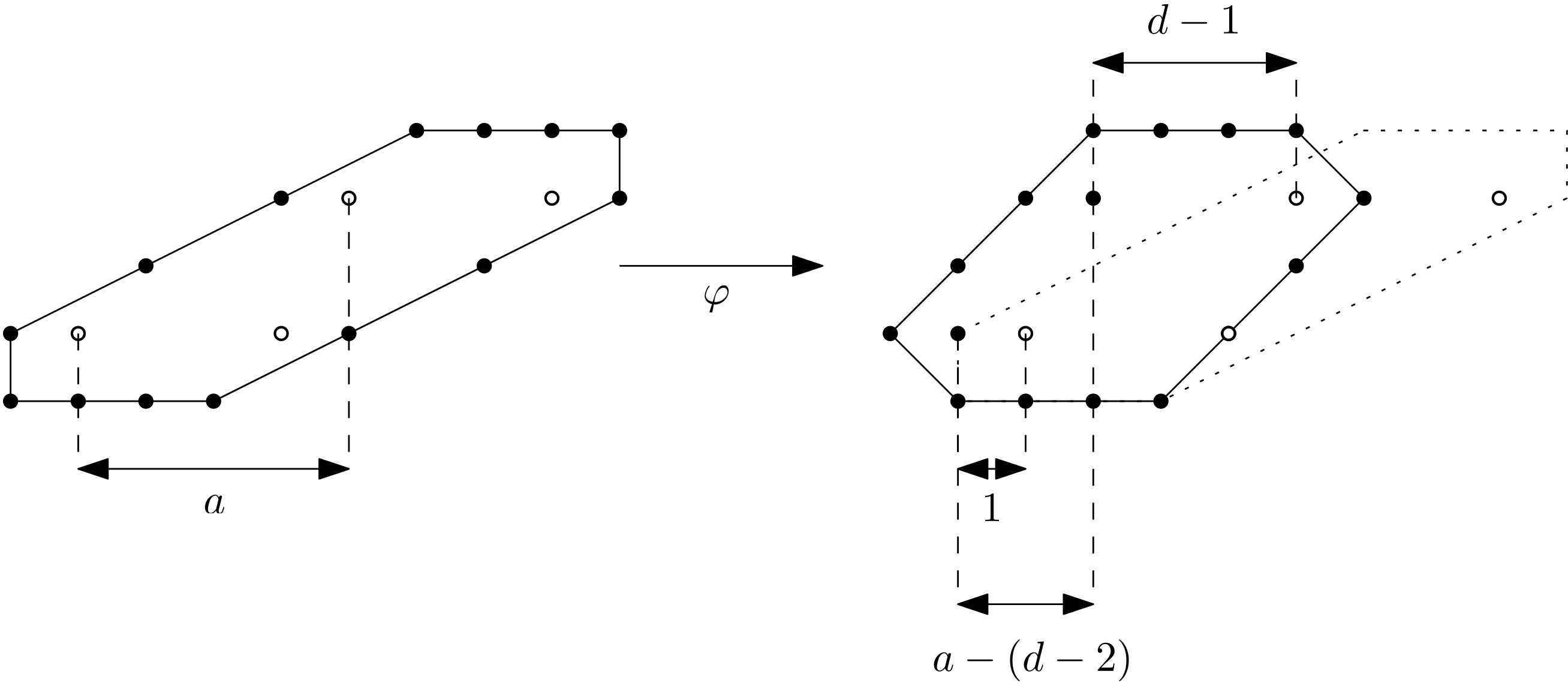}
            \caption{A shear transformation given by $\ell = -1$, where $\tau_1$ and $\tau_k$ relatively shift by $d - 2$.}
            \label{figure:shear}
        \end{figure}

    \begin{lemma}\label{lma:hascrystal}
    Fix $d \geq 3$ and $g \geq d^3$. Then, any polygon $P$ with genus $g$ and $\lw\left(P\right) = d$ is equivalent to a polygon that contains a crystal of length $d$. 
    \end{lemma}

    \begin{proof}
        We prove the contrapositive: if $P$ is not equivalent to some $P'$ that contains a crystal of length $d$, then $g < d^3$. Let $a$ be the length of $\tau_k^{(-1)}$ and $b$ the length of $\tau_1^{(-1)}$. Let $s_i$ be the signed slope\footnote{So long as $\tau_i$ is non-degenerate (as $\tau_2$, $\tau_n$ are, in the case of this proof) and non-vertical, $s_i$ is the unique value such that $\tau_i$ is a subset of the locus defined by $y = s_ix + b_i$.} of $\tau_{i}^{(-1)}$. If the face is vertical, then define $s_i = \infty$. Figure~\ref{figure:snconstruction} illustrates these labels.
    
        \begin{figure}[hbt]
            \centering
            \includegraphics[width=0.5\linewidth]{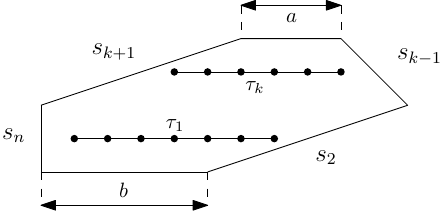}
            \caption{Faces $\tau_k, \tau_1$ and slopes $s_{k+1}$, $s_{k-1}$, $s_2$, $s_n$. Note that $s_n = \infty$.}
            \label{figure:snconstruction}
        \end{figure}

        By Proposition \ref{prop:crystalcondition}, the absence of a crystal imposes a condition on the number of interior boundary points lying on $\tau_k$ and $\tau_1$. Observe that these quantities are equal to $\ceil{-1/s_{k-1}} + \ceil{1/s_{k+1}} + (a-1)$ and $\ceil{-1/s_n} + \ceil{1/s_2} + (b-1)$ respectively. Without loss of generality, suppose $\tau_1$ is equal to or shorter than $\tau_k$, so that, by Proposition~\ref{prop:crystalcondition},
        \begin{equation}\label{eq:countinteriorpoints}
          \ceil{-1/s_{n}} + \ceil{1/s_{2}} + (b-1) < 2d - 2.
        \end{equation}
         Then, we can extend the cut lines that define $s_{n}$ and $s_{2}$ to form a trapezoid of height $d$ with side lengths $b$ and $b - d/s_{n} + d/s_{2}$. The area of this trapezoid upper bounds the area of $P$, denoted $A(P)$. Furthermore, Pick's theorem provides the bound $g \leq A(P)$. Therefore,
        \begin{align*}
            g \leq A\left(P\right) &\leq d\left( \frac{b + b -d/s_{n} + d/s_{2}}{2} \right) \\
            &\leq d\left( \frac{2b +d(\ceil{-1/s_n} + \ceil{1/s_2})}{2} \right) \\
            &< d\left( \frac{2b +d(2d - b - 1)}{2} \right) & \text{Using }\eqref{eq:countinteriorpoints}\\
            &= d\left( \frac{2b +2d^2 - bd - d}{2} \right) \\
            &= d\left( \frac{2d^2 + b\left(2-d\right) - d}{2} \right) \\
            &< d^3. & d \geq 3
        \end{align*}
        Thus, we have shown the contrapositive.
    \end{proof}

    \begin{theorem}\label{thm:existgonequalegon}
        Let $d \geq 3$, and let $P$ be a non-hyperelliptic maximal polygon with $\lw(P) = \egon(P) = d$ and $g \geq d^3$. Then, there exists a triangulation $\Delta$ such that $\dim(\modspace{\Delta}) = \dim \left(\modspace{P}\right)$. Furthermore, any tropical curve $\Gamma$ arising from $\Delta$ has $\gon(\Gamma) = d$.
    \end{theorem}

    \begin{proof}
        By Lemma \ref{lma:hascrystal}, $P$ is equivalent to some $P'$ that contains a crystal of length $d$. By Lemma \ref{lma:existsbeehive}, $P'$ has a beehive triangulation $\Delta$ such that $\Delta$ maximizes the dimension of $\modspace{\Delta}$. In other words, $\dim(\modspace{\Delta}) = \dim (\modspace{P'}) = \dim(\modspace{P})$. By Lemma \ref{lma:crystalsn}, since $\Delta$ is a refinement of $\cS_\omega$ by construction in Lemma \ref{lma:existsbeehive}, any tropical curve $\Gamma$ arising from $\Delta$ has $\gon(\Gamma) = d$.
    \end{proof}
    
    Finally, we prove Proposition~\ref{prop:equaldimgreater}.
        
    \begin{proof}[Proof of Proposition~\ref{prop:equaldimgreater}]
        Let $P$ be a maximal polygon with $\dim(\modspace{P}) = \dim(\mndexp{g}{d})$.

        \begin{itemize}
            \item \underline{Case 1: $P \cong (d+1) \Sigma$.} As the genus of $d\Sigma$ is 
            \[g((d+1)\Sigma) = \frac{d(d-1)}{2} < d^3,\]
            this case violates the assumption that $g \geq d^3$.
            \item \underline{Case 2: $P \cong 2\Upsilon$.} 
            We observe that the genus of $2\Upsilon$ is 4, which also violates $g \geq d^3$.
            \item \underline{Case 3: $P$ is any other (maximal) polygon.} In this case, by Proposition~\ref{lwisegon}, $\lw(P) = \egon(P) = d$, so we may apply Theorem~\ref{thm:existgonequalegon} to argue the lower bound.
        \end{itemize}
    \end{proof}

\section{An Upper Bound on $\dim\left(\mathbb{M}_{g,\underline{d}}^\mathrm{nd}\right)$}
\label{sec:upper}

In this section, we establish an upper bound on the dimension of the expected gonality locus 
\[
\mathbb{M}_{g,\underline{d}}^{\mathrm{nd}} 
= \!\!\!\!\!\!\!\!\bigcup_{\substack{g(P) = g \\ \egon(P) = d}} \!\!\!\!\!\!\!\mathbb{M}_P.
\]
by obtaining an upper bound on $\dim(\mathbb{M}_P)$ for such polygons $P$. Recall from Lemma~\ref{lma:column_vector_formula} that if a non-hyperelliptic polygon $P$ is maximal then
\begin{equation}\label{eq:dmp}
    \dim(\mathbb{M}_P) = g(P) - 3+ r(P)  - c(P),
\end{equation}
where $g(P)$ is the number of interior lattice points of $P$, $r(P)$ is the number of boundary lattice points of $P$, and $c(P)$ is the number of column vectors. Moreover, this dimension is invariant under unimodular transformations of $P$.

Our strategy is to bound $r(P) - c(P)$ in terms of the number of $g(P)$ and $d$, which may be independently of interest. The main idea of this proof is to compare $P$ against a trapezoid\footnote{An earlier version of the manuscript defined $P_{\text{trunc}}$ to be a different shape.} $P_{\text{trunc}}$.

\begin{lemma}\label{lem:rPcPbound}
    Let $P$ be a lattice polygon with $g(P)$ interior points, $r(P)$ boundary points, $c(P)$ column vectors, and lattice width $d>1$. Then,
    \[r(P) - c(P) \leq \frac{2g(P)}{d-1} + 2d.\]
\end{lemma}

Combining Lemma~\ref{lem:rPcPbound} with \eqref{eq:dmp} yields the desired upper bound:

\begin{theorem}\label{lma:maxdim2}
    Let $P$ be a maximal lattice polygon with $g$ interior points and lattice width $d>1$. Then,
    \begin{equation}\label{eq:maxdim2bound}
        \dim\!\left(\mathbb{M}_{P}\right) \;\leq\; g -3 + \frac{2g}{d-1} + 2d.
    \end{equation}
\end{theorem}

\subsection{Proving Lemma~\ref{lem:rPcPbound}}

    In this section, we prove Lemma~\ref{lem:rPcPbound}. We set some notation. Without loss of generality, let $\tau_a$ be the bottom horizontal face of $P$ that lies on $y = 0$, and let $\tau_b$ be the top horizontal face that lies on $y = d$. Let their lengths be $a$ and $b$ respectively. Let $P_{\text{trunc}} = \conv\{{\tau_a, \tau_b}\}$. Call the region to the left of $P_{\text{trunc}}$ the \emph{left cap} $P_L$, and similarly define the right cap $P_R$. Note that $P_L$ and $P_R$ may be degenerate. See Figure~\ref{figure:ptrunc}.

    \begin{figure}[hbt]
       \centering
       \includegraphics[width=0.4\linewidth]{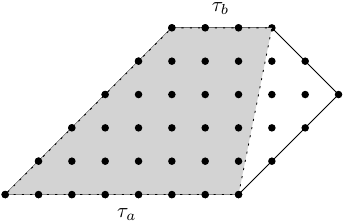}
        \caption{A polygon $P$ and its truncation $P_{\text{trunc}}$ (shaded). Here, $P_L$ is degenerate, and $P_R$ consists of the remaining region in white.}
        \label{figure:ptrunc}
    \end{figure}

    Let the \emph{left boundary} be the path from the left end of $\tau_a$ to the left end of $\tau_b$ that travels along the left edges of $P$. The left boundary chain may hit at most $d + 1$ lattice points. Let $\delta_L$ denote the number of ``missed'' points. We define the right boundary chain and $\delta_R$ analogously. Define $\delta := \delta_L + \delta_R,$ and observe that
    \begin{equation}\label{eq:rformula}
    r(P) = a + b + 2d - \delta.  
    \end{equation}

    As shorthand, we use $E := A(P_L) + A(P_R)$ to denote the sum of areas. Observe that
    \begin{equation}\label{eq:Aformula}
    A(P) = \frac{(a+b)d}{2} + E.       
    \end{equation}    
    We first solve for an explicit formula for $a + b$. 
    \begin{proposition}
        Assuming the hypotheses of Lemma~\ref{lem:rPcPbound},
        \[a + b  = \frac{2g}{d-1} -\frac{2E + \delta}{d-1} + 2.\]
    \end{proposition}

    \begin{proof}
            Slightly rearranging Pick's theorem, and plugging in \eqref{eq:rformula} and \eqref{eq:Aformula},
    \begin{align*}
        2g &= 2A - r + 2 \\
        &= \left((a+b)d + 2E\right) - (a+b + 2d - \delta) + 2.
    \end{align*}
    Rearrange to find the result.
    \end{proof}

    The main step is the following estimate:
    \begin{proposition}
        Assuming the hypotheses of Lemma~\ref{lem:rPcPbound},
        \[\frac{2E + d\delta}{d-1} + c(P)\geq 2.\]
    \end{proposition}

    \begin{proof}
        Let us define
        \[
        \chi_L = \begin{cases}
            1 & \text{if $(-1,0)$ is a column vector of $P$} \\
            0 & \text{otherwise}
        \end{cases},
        \]
        and similarly define $\chi_R$ with the column vector $(1,0)$. Note that $c(P) \geq \chi_L + \chi_R$.
    
        It suffices to show that \[\frac{2A(P_L) + d\delta_L}{d-1} +\chi_L \geq 1,\] as the same argument can be applied to the right cap $P_R$. We can loosely interpret this statement as the \emph{loss in dimension} (over the general statement of \eqref{eq:dmp}) of the moduli space due to the left cap.

        We split into cases:
        
        \begin{itemize}
            \item \underline{Case 1: $\delta_L \geq 1$.} In this case,
           \[\frac{2A(P_L) + d\delta_L}{d-1} + \chi_L\geq \frac{ d\delta_L}{d-1} \geq 1.\]
            \item \underline{Case 2: $\delta_L = 0$ and $P_L$ is degenerate.} In this case, $(-1,0)$ is a column vector, so $\chi_L = 1$. This is sufficient.
            \item \underline{Case 3: $\delta_L = 0$ and $P_L$ is not degenerate.} 
            Observe that the left boundary contains at least $d+1$ boundary points, as $\delta_L = 0$ implies that none of the $d+1$ boundary points are missed. By Pick's theorem,
            \[2A(P_L) = 2g(P_L) + r(P_L) - 2 \geq d-1,\]
            which is sufficient.
        \end{itemize}
        Combining the results for the left and right caps finishes the proof.
    \end{proof}

    Using this result, Lemma~\ref{lem:rPcPbound} follows from a direct computation using the above results:
\begin{align*}
    r(P) - c(P)
    &= a + b + 2d - \delta - c(P) \\
    &= \frac{2g}{d-1} - \frac{2E+\delta}{d-1}
       + 2 + 2d - \delta - c(P) \\
    &= \frac{2g}{d-1} + 2d + 2
       - \left(\frac{2E+d\delta}{d-1} + c(P)\right) \\
    &\leq \frac{2g}{d-1} + 2d.
\end{align*}
This yields the statement of Lemma~\ref{lem:rPcPbound}.

\subsection{Proving Theorem ~\ref{lma:maxdim2_restate}} What remains is to account for the case when $\lw(P)\neq\egon(P)$:
\begin{lemma}\label{lma:egonnotlw}
Theorem~\ref{lma:maxdim2_restate} holds when $\lw(P)\neq\egon(P)$.
\end{lemma}
\begin{proof}
    Assuming $\lw(P)\neq\egon(P)=d>1$, then by Proposition~\ref{lwisegon}, either $P\cong (d+1)\Sigma$ or $d=3$ and $P\cong2\Upsilon$. In both situations, we compute dimension directly, via Lemma~\ref{lma:column_vector_formula}:
    
    \begin{itemize} \item \underline{Case 1: $P\cong(d+1)\Sigma$.} Then $g(P)=d(d-1)/2$, $r(P)=3(d+1)$, and $c(P)\geq3$. By Lemma~\ref{lma:column_vector_formula},
    \begin{align*}
        \dim\left(\mathbb{M}_{(d+1)\Sigma}\right)&=g(P)+r(P)-3-c(P)\\
        &\leq g(P)+3d-3\\
&= g(P)+\frac{d(d-1)}{d-1}+2d-3\\
&= g(P)+\frac{2g(P)}{d-1}+2d-3.
    \end{align*}
    
\item \underline{Case 2: $d=3$ and $P\cong2\Upsilon$.} Then $g(P)=4$, $r(P)=6$, and $c(P)=0$. By Lemma~\ref{lma:column_vector_formula},
\begin{align*}\dim\left(\mathbb{M}_{P}\right)&=g(P)+r(P)-3-c(P)\\
&=7\\
&\leq g(P)+\frac{2g(P)}{d-1}+2d-3\\
&=11.\end{align*}
\end{itemize}
\end{proof}

\begin{proof}[Proof of Theorem~\ref{lma:maxdim2_restate}] This follows from Theorem~\ref{lma:maxdim2} and Lemma~\ref{lma:egonnotlw}.
\end{proof}

\section{Equality of $\dmndexp{g}{d}$ and $\dmnd{g}{d}$ for large genus}\label{sec:equal-dim}

Building on the previous section, we now prove Proposition~\ref{prop:equaldimless}—that for a fixed $d\geq 3$ and a sufficiently large genus $g\geq \max\{d^3,32\}$, the dimension of $\mndexp{g}{d}$ upper-bounds the dimension of $\mnd{g}{d}$. Combining this result with that of Proposition~\ref{prop:equaldimgreater}, we obtain the equality \[\dmnd{g}{d}=\dmndexp{g}{d}.\]

To gain intuition behind the proof of Proposition~\ref{prop:equaldimless}, suppose $\Gamma\in \mnd{g}{d}$ for some genus $g\geq 3$ and gonality $d\geq 3$. Further suppose that $\Gamma$ arises from the polygon $P$ of expected gonality $\egon(P)=d'$. Lemma~\ref{lma:gon_at_most_egon_graphs} gives $d\leq d'$, and thus
\[\mnd{g}{d}\mathrel{\subset}\!\bigcup_{d'\geq d}\mndexp{g}{d'}.\]

Furthermore, we claim that the above union is finite by establishing Lemma~\ref{lma:maxlatticewidth}, which states that the lattice width (and thus expected gonality) of a maximal polygon of genus $g$ cannot exceed $2\sqrt{g+2}$. This provides an upper bound on $\dmndc{g}{d}$. Explicitly, writing $\left\llbracket d,2\sqrt{g+2}\right\rrbracket$ for $\Z\cap\left[ d,2\sqrt{g+2}\right],$
\begin{equation}\label{eq: reducedProp}
\dmnd{g}{d}\leq \dim\left(\bigcup_{d'\in\left\llbracket d,2\sqrt{g+2}\right\rrbracket}\!\!\!\!\!\!\!\!\!\!\mndexp{g}{d'}\right)\leq\max_{d'\in\left\llbracket d,2\sqrt{g+2}\right\rrbracket} \dmndexp{g}{d'}.
\end{equation}
Thus with Lemma~\ref{lma:maxlatticewidth}, Proposition~\ref{prop:equaldimless} reduces to proving that for a fixed expected gonality $d\geq 3$ and genus $g\geq \max\{d^3,32\}$, \begin{equation}\label{eq: goal} \dmndexp{g}{d'}\leq\dmndexp{g}{d} \text{ for all } d'\in\left\llbracket d,2\sqrt{g+2}\right\rrbracket. \end{equation}

\begin{lemma}\label{lma:maxlatticewidth} Let $P$ be a maximal lattice polygon of genus $g$ and lattice width $d$. Then
    \[d \leq 2\sqrt{g + 2}.\]
\end{lemma}

\begin{proof}
    Let $P$ be a polygon. By \cite[Theorem 12]{Sop23}, we have the inequality \begin{equation}\label{eq: latticesize} \frac{3}{8}\ls\left(P\right)\lw\left(P\right) \le A\left(P\right).\end{equation} Here, $\ls\left(P\right)$ is the lattice size of $P$—as defined in Section~\ref{sec:background}.  By Pick's theorem, and the fact that $d = \lw\left(P\right) \le \ls\left(P\right)$, \eqref{eq: latticesize} expands to
    \begin{equation}\label{eq: picksandlatticesize}\frac{3}{8}d^2\leq \frac{3}{8}\ls\left(P\right)\lw\left(P\right)\leq A(P)=g+\frac{r}{2}-1.\end{equation}
In \cite{moving_out}, the author shows that for polygons not equivalent to $d\Sigma$,  
    \[r \leq r^{\left(1\right)}+8,\] 
    where $r^{(1)}$ is the number of boundary points of the interior polygon. Certainly $r^{(1)}\leq g$, and so we obtain the inequalities \begin{equation}\label{eq: rbound} g + \frac{r}{2} - 1 \leq g+\frac{r^{(1)}+8}{2}-1\leq g+\frac{g+8}{2}-1\leq\frac{3}{2}g + 3.\end{equation} Combining \eqref{eq: picksandlatticesize} and \eqref{eq: rbound}, we find that $$\frac{3}{8}d^2\leq\frac{3}{2}g+3,$$ which occurs precisely when $$d \le 2\sqrt{g+2}.$$
    On the other hand, if $P\cong d\Sigma$, then $g(P)=(d-1)(d-2)/2$, and
    \begin{align*}
d \le 2\sqrt{\frac{(d-1)(d-2)}{2}+2}
&\iff d^2 \le 4\!\left(\frac{(d-1)(d-2)}{2}+2\right)\\
&\iff d^2 \le 2\bigl(d^2-3d+2\bigr)+8\\
&\iff 0 \le d^2-6d+12\\
&\iff 0 \le (d-3)^2+3,
\end{align*}
which is true.
\end{proof}

\subsection{The upper bound function} Define the function $U:\R_{>0}\times\R_{>1}\to\R$ by $$\left(g,d\right)\mapsto g+\frac{2g}{d-1} + 2d-3.$$ Then Theorem~\ref{lma:maxdim2_restate} is equivalent to the inequality 
\begin{equation}\label{eq: bone}
\dmndexp{g}{d} \leq {\bone{g}{d}}.
\end{equation}
for all integers $g>0,d> 1$.

The following lemmas will organize the properties of $U$ that will be referenced throughout the proof of Proposition~\ref{prop:equaldimless}.

\begin{lemma}\label{lma:pospartial}

The second partial derivative of $U$ with respect to $d$ is positive on $\R_{>0} \times \R_{> 1}$. Specifically, one may check that
\[
\frac{\partial^2 U}{\partial d^2}\left(g,d\right)  = \frac{4g}{\left(d - 1\right)^3} > 0.
\]

\end{lemma}

\begin{lemma}\label{lma:b1bound2}
Fix \(d\ge 3\). Then for $g \ge \max\{d^3,\,32\}$,
\[
\bone{g}{2\sqrt{g+2}} \;\le\; \bone{g}{d+1}.
\]
\end{lemma}

\begin{proof} A direct subtraction gives
\[
\begin{aligned}
\Delta
&:=\bone{g}{d+1}-\bone{g}{2\sqrt{g+2}} \\
&=2g\!\left(\frac{1}{d}-\frac{1}{\,2\sqrt{g+2}-1\,}\right)
+2\!\left(d-(2\sqrt{g+2}-1)\right)\\
&=2\bigl((2\sqrt{g+2}-1)-d\bigr)
\left(\frac{g}{d(2\sqrt{g+2}-1)}-1\right).
\end{aligned}
\]
Since $g\ge d^3$ and $d\ge 3$, we have $2\sqrt{g+2}-1-d>0$. Hence $\Delta\ge 0$ is equivalent to
\begin{equation}\label{eq:star}
\frac{g}{\,2\sqrt{g+2}-1\,}\;\ge\; d
\quad\Longleftrightarrow\quad
g \;\ge\; d\bigl(2\sqrt{g+2}-1\bigr).
\end{equation}
Define \(F(g):=g-d(2\sqrt{g+2}-1)\). Then
\[
F'(g)\;=\;1-\frac{d}{\sqrt{g+2}}
\;\ge\;1-\frac{d}{\sqrt{d^3+2}}
\;\ge\;1-\frac{1}{\sqrt d}\;>\;0
\qquad(d\ge 3),
\]
so $F$ is increasing on $[d^3,\infty)$. Thus it suffices to check $F(g)\ge 0$ at the left endpoint.

\smallskip
\begin{itemize}
\item \underline{Case 1: $d\ge 4$.}
At $g=d^3$, condition \eqref{eq:star} becomes
\[
d^3 \;\ge\; d\bigl(2\sqrt{d^3+2}-1\bigr)
\;\Longleftrightarrow\;
(d^2+1)^2 \;\ge\;4(d^3+2),
\]
i.e.
\[
P(d):=d^4-4d^3+2d^2-7\;\ge\;0.
\]
For $d\ge 4$, $P'(d)=4d(d^2-3d+1)>0$, and $P(4)=25>0$, so $P(d)\ge 0$. Hence \eqref{eq:star} holds for all $g\ge d^3$.

\item \underline{Case 2: $d=3$.}
Condition \eqref{eq:star} is
\[
g \;\ge\; 3\bigl(2\sqrt{g+2}-1\bigr)
\;\Longleftrightarrow\;
\sqrt{g+2}\;\ge\;3+\sqrt{8}.
\]
This is equivalent to $g\ge (3+\sqrt{8})^2-2=15+6\sqrt{8}\approx 31.97$, so any integer $g\ge 32$ works.
\end{itemize}

Combining the two cases gives the claim.
\end{proof}

\begin{lemma}\label{lma:b1bound1}
    Fix $d\geq 3$. Then for any genus $g \geq d^3$,
    \[\bone{g}{d+1}\leq\dmndexp{g}{d}.\]
\end{lemma}

\begin{proof}

We split into cases: first, let us \underline{assume $d = 3$}. We supply an explicit proof that
    \[
        \dmndexp{g}{3} = \begin{cases}
            6 & g = 3 \\
            2g + 1 & \text{otherwise} \\
        \end{cases}.
    \]
in Theorem~\ref{thm: d=3}; one may check that the inequality holds when $g \geq 27$.

Now, we deal with \underline{the $d \geq 4$ case}. We give an explicit construction of $P$, which takes the form of a ``cut'' rectangle depicted in Figure~\ref{fig:ex-chop} such that $U(g,d+1) \leq \dim(\mathbb{M}_P)$. Define integers $m, k$ where $m > 0$ and $0 \leq k < d-1$ such that
\begin{equation}
    \label{eq:chop-rect-genus}
    g = (m-1)(d-1) - k.
\end{equation}
As a result of Gauss \cite{andrews1986eureka}, $k$ may be written as a sum of three triangular numbers. We index the triangular numbers via the following scheme
\[
T_j = \begin{cases}
    0      & \text{if } j = 0, \\
    1      & \text{if } j = 2, \\
    3      & \text{if } j = 3, \\
    \vdots & \vdots
\end{cases}.
\]
Decompose $k$ into three triangular numbers $T_{x_1}, T_{x_2}, T_{x_3}$.

Let $P$ be the rectangle constructed by ``chopping off'' the corners $v_i$ of a $m \times d$ rectangle. Specifically, we cut an isosceles triangle of leg-length $x_i$ from corner $v_i$ whenever $x_i > 1$ for $i \in \{1,2,3\}$. We say this is a \emph{cut of length $x_i$}. Observe that $T_j$ represents how many interior points we ``lose'' by making a cut of length $x_j$. In particular, we do not allow $j = 1$ because that does not create a maximal polygon.

We claim that $P$ is maximal; the only ``interesting'' case to check is $x_i = d-1$ (which would result in a non-maximal polygon). This is not possible because
\[T_{d-1} = \frac{(d-2)(d-1)}{2} \geq d-1 > k\]
when $d \geq 4$. Furthermore, the cuts do not overlap with each other: on the ``short side'' of $R$, for $i \neq j$, we have $(x_i-1)+(x_j-1)\leq T_{x_i}+T_{x_j}\leq k\leq d-2$, and thus $x_i+x_j\leq d$.

\begin{figure}[hbt]
        \centering
    \includegraphics[width=0.4\textwidth]{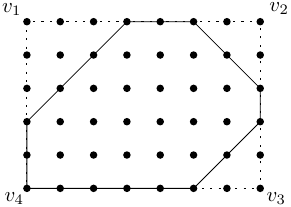}
        \caption{An example cut rectangle $P$, with $x_1 = 3$, $x_2 = 2$, $x_3 = 2$.}  \label{fig:ex-chop}
\end{figure}

One may check that $P$ has genus $g$ as in \eqref{eq:chop-rect-genus}. Additionally, one may argue that $P$ has lattice width $d$ by locating a $d \times d$ square inside of $P$.

We claim that, in terms of $g, d$ and the cuts $(x_i,y_i)$, we can compute:
\begin{equation}\label{eq: boundr}
r = \frac{2g}{d-1} + 2 + 2d +\sum_{i=1}^3X_i,\end{equation}
where the $X_i$'s are defined as:
\begin{equation}\label{eq: xi}
    X_i := \frac{x_i \left(x_i -  d\right)}{d-1}.
\end{equation}
We sketch an argument. The cut rectangle has $r = 2m + 2d - \sum_{i=1}^3 x_i$ boundary points. Since \eqref{eq:chop-rect-genus} relates expresses $m$ in terms of $g, d,$ and the $x_i$'s, one may rewrite this expression by eliminating $m$.
 
We now show an intermediate step:
\begin{equation}
\label{eq:gap}
\bone{g}{d}-\dim\left(\modspace{P}\right)\leq d+3.
\end{equation}

By Lemma~\ref{lma:column_vector_formula} and \eqref{eq: boundr},
\begin{equation}\label{eq:cutRectangleBound}
    \dmodspace{P}=g+\frac{2g}{d-1}+2+2d+\sum_{i=1}^3X_i-c\left(P\right)-3,\end{equation}
 Subtracting from $U(g,d)$ yields
\begin{equation}\label{eq: comparing1}\bone{g}{d}-\dmodspace{P}=-\sum_{i=1}^3X_i-2+c\left(P\right).\end{equation}
Since $P$ is a cut rectangle, solely with short cuts, we have that $c\left(P\right)\leq 4$. Furthermore, since $\sum_{i=1}^3X_i$ simplifies to $\frac{\sum_{i=1}^3x_i^2-x_id}{d-1}$, \eqref{eq: comparing1} reduces to
\begin{equation}\label{eq: comparing2} \bone{g}{d}-\dmodspace{P} \leq -\frac{\sum_{i=1}^3x_i^2-x_id}{d-1}+2.\end{equation}
Observe that the RHS is maximized when $x_i = d/2$. Thus,
\begin{equation}\label{eq: comparing3} \bone{g}{d}-\dmodspace{P} \leq \frac{3d^2}{4(d-1)}+2 \leq d+3.\end{equation}
Finally, we prove that for $g\geq d^3$, \begin{equation}\label{eq: gap}\bone{g}{d+1}\leq \bone{g}{d}-(d+3).\end{equation} We will show the gap is non-negative.
\begin{align*}
(\bone{g}{d} - \bone{g}{d+1}) - (d+3) &= 2g\left(\frac{1}{d-1} - \frac{1}{d} \right)  - d- 5 \\
&\geq \frac{2d^2}{d-1}-d -5 .
\end{align*}
One may verify that the right-hand side is non-negative for all positive integers $d$.
This gives \eqref{eq: gap}, which in conjunction with \eqref{eq:gap}, proves that for $g\geq d^3$,
\[\bone{g}{d+1}\leq\bone{g}{d}-\left(d+3\right)\leq \dmodspace{P}\leq\dmndexp{g}{d}.\]
\end{proof}

\subsection{Proving Theorem~\ref{tmd:equaldim}.}
\begin{proposition}
    \label{prop:equaldimless}
    Fix $d\geq 3$. Then for all genera $g 
\geq \max\{d^3,32\}$, 
    \[\dim\Bigl(\mnd{g}{d}\Bigr) \leq \dmndexp{g}{d}.\]
\end{proposition}

\begin{proof}

We prove \eqref{eq: goal}—that  $\dmndexpc{g}{d'}\leq\dmndexpc{g}{d}$ for all $d'\in\left\llbracket d,2\sqrt{2+g}\right\rrbracket$. Along with Lemma~\ref{lma:pospartial}, which states that the second partial derivative of $\bone{g}{d}$ with respect to $d$ is positive, Lemma~\ref{lma:b1bound2} implies that \begin{equation}\bone{g}{d'}\leq \bone{g}{d+1}\text{ for all } d'\in\left[d+1,2\sqrt{g+2}\right]\supset\left\llbracket d+1,2\sqrt{g+2}\right\rrbracket.
\label{eq: beatsall}
\end{equation}
Prepending \eqref{eq: bone} to \eqref{eq: beatsall}, and then appending Lemma~\ref{lma:b1bound1}, we conclude that
\[\dmndexp{g}{d'}\leq\dmndexp{g}{d} \text{ for all } d'\in\left\llbracket d,2\sqrt{g+2}\right\rrbracket.\]
For visual intuition, see Figure~\ref{fig:graphU}.
\end{proof}
\begin{figure}[hbt]
        \centering
    \includegraphics[width=\linewidth]{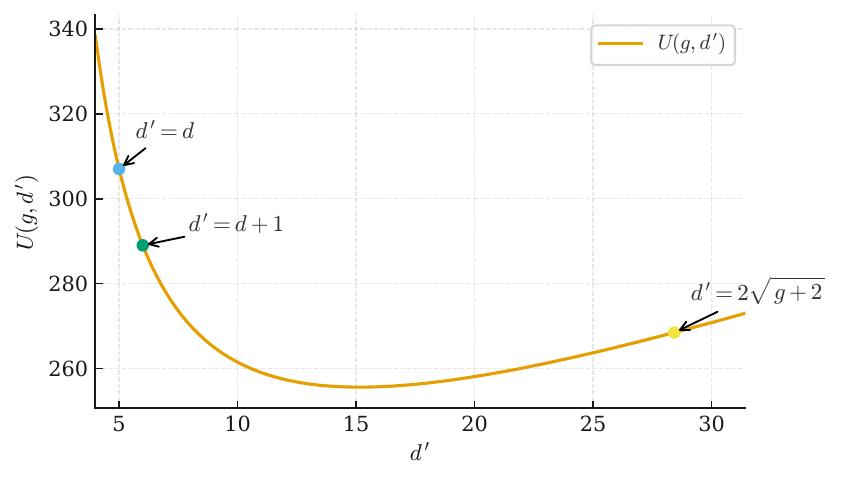}
        \caption{The graph of $U(g,d')$ with $g=200$ fixed and $d=5$.}  \label{fig:graphU}
    \end{figure}

\begin{proof}[Proof of Theorem~\ref{tmd:equaldim}.]
    The result follows from Propositions $\ref{prop:equaldimgreater}$ and $\ref{prop:equaldimless}$.
\end{proof}

\medskip
\noindent
\subsection{Future directions.} 
A natural next step is to relax the high-genus assumption in Theorem~\ref{tmd:equaldim} and prove Conjecture~\ref{conj:equalmodspace} for all genera $g$ and gonalities $d$. Such a proof, if combinatorial, would likely require an invariant that provides a tighter bounds on gonality than the scramble number, as the authors were unable to fully solve the $g = 7$ case of Theorem~\ref{tmd:equaldim} via scramble number. Ultimately, this line of study aims to resolve Conjecture~\ref{conj:gon=egon}.

A related question is to study the geometric structure of the gonality locus $\mnd{g}{d}$. While the expected gonality locus $\mndexp{g}{d}$ is a union of polyhedral cones, it remains an open question whether $\mnd{g}{d}$ admits such a characterization.


\section{Dimensions of moduli spaces for small genus and lattice width}\label{sec:dimmodspace_lowgd}

In this section, we compute the dimensions of $\mnd{g}{d}$ and $\mndexp{g}{d}$ for certain small values, complementing our main result, Theorem~\ref{tmd:equaldim}, in the small genus cases. The first half of this section will focus on stratifying by low genus, and the remainder focuses on stratifying by low expected gonality.  We remark that whenever we have a formula for the dimension of $\mnd{g}{\underline{d}}$, the same formula holds for the corresponding moduli space $\mathcal{M}_{g,d}^\mathrm{nd}$ of non-degenerate algebraic curves of genus $g$ and gonality $d$.

\subsection{Dimensions in low genus}\label{sec:dimlowg}

We show that $\dim(\mnd{g}{d}) = \dim(\mnd{g}{\underline{d}})$ when $g$ is small: 

\begin{theorem}\label{thm:low_genus_computation}
For $g \leq 6$ and $g = 8$,
\[
\dmnd{g}{d} = \dmndexp{g}{d}.
\]
The values of these dimensions are summarized in Table \ref{table:dimensions_low_genus}.
\end{theorem}

\begin{table}[hbt]
\begin{tabular}{|c|c|c|c|c|c|c|}
\hline
      & $g=2$ & $g=3$ & $g=4$ & $g=5$ & $g=6$ & $g=8$ \\ \hline
$d=2$ & $3$   & $5$   & $7$   & $9$   & $11$  & $15$  \\ \hline
$d=3$ &       & $6$   & $9$   & $11$  & $13$  & $17$  \\ \hline
$d=4$ &       &       &       & $10$  & $13$  & $17$  \\ \hline
\end{tabular}
\caption{The dimension of $\mathbb{M}_{g,d}^\textrm{nd}$ (and of $\mathbb{M}_{g,\underline{d}}^\textrm{nd}$, and of $\mathcal{M}_{g,d}^\textrm{nd}$) for small values of $g$.}
\label{table:dimensions_low_genus}
\end{table}

Below, we present the $g = 5$ case. The remaining proofs can be found in Appendix~\ref{sec: tropical_Appendix} and follow a similar flavor.

\begin{proof}
    For our analysis of genus 5 polygons, we exclude $d = 1$ because a polygon has expected gonality $1$ if and only if its genus is $0$. The case $d = 2$ is covered by Theorem~\ref{thm:d_is_2}. 
    
    Since hyperelliptic polygons only give rise to hyperelliptic metric graphs \cite[Theorem 1.1]{tropical_hyperelliptic_curves_in_the_plane}, we only need to consider non-hyperelliptic polygons of genus $5$. Furthermore, recall that every metric graph arises from a maximal polygon \cite[Lemma 2.6]{brodsky_joswig_morrison_sturmfels}; thus our enumeration may be restricted to those polygons which are maximal.

    Explicit computations of the maximal non-hyperelliptic polygons can be found in the linked material for \cite{cape2025tropicalcrossingnumberfinite} for $g \leq 25$.\footnote{\url{https://github.com/noahcape/tcn}} For $g = 5$, there are four such polygons up to equivalence, as shown in Figure~\ref{figure:maximal_polygons_genus_5}. Calling these $P_1^{5}$, $P_2^{5}$, $P_3^{5}$, and $P_4^{5}$ from left to right, their expected gonalities are $\egon(P_1^{5})=3$ and $\egon(P_i^{5})=4$ for $2\leq i\leq 4$. Since no polygon has expected gonality $5$ or higher, we only need to consider the cases $d = 3$ and $d = 4$. 

    \begin{figure}[hbt]
        \centering
    \includegraphics{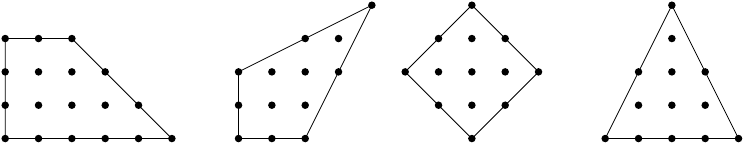}
        \caption{The four maximal non-hyperelliptic polygons of genus $5$ up to equivalence.}  \label{figure:maximal_polygons_genus_5}
    \end{figure}

    By Lemma~\ref{lma:column_vector_formula}, one can calculate the moduli dimensions of $P_1^{5}$, $P_2^{5}$, $P_3^{5}$, and $P_4^{5}$ to be $11$, $10$, $10$, and $9$, respectively. These values have been computationally verified in \cite[Theorem 8.5]{brodsky_joswig_morrison_sturmfels}. These computations immediately show $\dim(\mathbb{M}_{5,\underline{3}}^{\textrm{nd}})= 11$ and $\dim(\mathbb{M}_{5,\underline{4}}^{\textrm{nd}}) = 10$. 

    By Theorem~\ref{thm:d_is_3}, since $P_1^{5}$ can only give rise to metric graphs of gonality $3$, we have $\dim(\mnd{5}{3}) = 11$ as the ambient space $\mathbb{M}^{\textrm{nd}}_5$ is $11$-dimensional.

    \begin{figure}[hbt]
        \centering
        \includegraphics[width=0.8\linewidth]{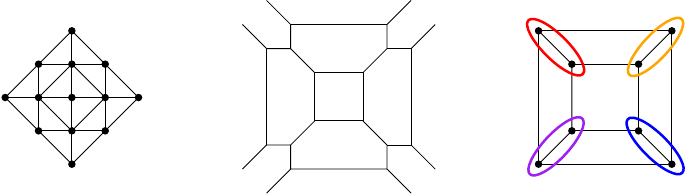}
        \caption{A beehive triangulation of $P_3^{5}$, a tropical curve arising from the triangulation, and its skeleton with a scramble of order $4$ (circled pairs of vertices form the eggs).}
        \label{figure:p3_scramble}
    \end{figure}

    Lemma~\ref{lma:gon_at_most_egon_graphs} shows that metric graphs of gonality $4$ could arise only from $P_2^{5}$, $P_3^{5}$, and $P_4^{5}$. Thus, $\dim(\mathbb{M}_{g,{4}}^{\textrm{nd}})\leq 10$. To show that $\dim(\mathbb{M}_{5,{4}}^{\textrm{nd}})= 10$, we present a $10$-dimensional family of metric graphs of gonality $4$ arising from a beehive triangulation of $P_3$, as illustrated by Figure~\ref{figure:p3_scramble}. As argued at the end of Section~\ref{sec:background}, all these graphs have gonality $4$; recall we obtained a lower bound using the scramble illustrated on the right.
    
    Thus, by Proposition~\ref{prop:snlowerbound}, every metric graph arising from this triangulation has gonality exactly $4$. By Lemma~\ref{lma:beehive}, the space of graphs arising from this triangulation is $10$-dimensional, and $\dim(\mathbb{M}_{5,{4}}^{\textrm{nd}})= 10$.

    Finally, for $d \geq 5$, the space $\mndexp{g}{d}$ is empty because there do not exist genus $5$ polygons of expected gonality at least $5$. Similarly, the space $\mnd{g}{d}$ is also empty by Lemma~\ref{lma:gon_at_most_egon_graphs}.
\end{proof}

\subsection{Dimensions in low expected gonality}

In this section we describe $\dmndexpc{g}{d}$ for low values of $d$, which also gives rise to the respective values of $\dim\bigl(\mathcal{M}_{g,d}^\mathrm{nd}\bigr)$. We first review some results that directly follow from the literature when $d = 2$ and $d =3$.

\begin{theorem}[{\cite[Theorem 6.1]{brodsky_joswig_morrison_sturmfels}}]
\label{thm:d=2}
    Let $g \geq 2$. Then
    \[
        \dmndexp{g}{2} = 2g - 1.
    \]
\end{theorem}

\begin{proof}
    This is a restatement of \cite[Theorem~6.1]{brodsky_joswig_morrison_sturmfels}.
\end{proof}

\begin{theorem}\label{thm: d=3}
        Let $g \geq 3$. Then
    \[
        \dmndexp{g}{3} = \begin{cases}
            6 & g = 3 \\
            2g + 1 & \text{otherwise} \\
        \end{cases}.
    \]
\end{theorem}

\begin{proof}
    The lower bound for the $g = 3$ case is done via identifying $4\Sigma$, which \cite{brodsky_joswig_morrison_sturmfels} argues has $\dim(\mathbb{M}_{4\Sigma}) = 6$. 
    The upper bound follows from \cite[Theorem 1.1]{brodsky_joswig_morrison_sturmfels} since $\mndexp{g}{3} \subseteq \mathbb{M}_{g}^{\text{nd}}$ (though they use $\mathbb{M}_{g}^{\text{planar}}$ to denote this space).

    For all other $g > 3$, the lower bound of $2g+1$ follows from the proof of the main theorem of \cite[Theorem 1.1]{brodsky_joswig_morrison_sturmfels}, where their constructions in Equations 17 and 18 satisfy $d = 3$. The upper bound follows from the well-known fact that $\dim(\mathcal{M}_{g,3}) = 2g + 1$ in \cite[Proposition 1.3]{MR2443348} (though they call this $\mathcal{T}_g$), and $\mathcal{M}^{\text{nd}}_{g,3} \subseteq \mathcal{M}_{g,3}$.
\end{proof}

Next, using our methods, we compute the $d = 4$ case when $g \equiv 1 \mod 3$ and estimates on the $d=5$ case.

\begin{theorem}\label{thm: d=4}
Let $g \geq 7$ and $ g\equiv 1\pmod 3$. Then,
\[\dmndexp{g}{4}=\floor{U(g,4)} = \floor{\frac{5g}{3}} + 5.\]
\end{theorem}
\begin{proof}
Consider the cut rectangle $P$ obtained from the $[(g-1)/3+2]\times 4$ rectangle by two antipodal cuts of lengths $2,2$. See Figure~\ref{fig: d=4} for the example of $P$ when $g=7$.

\begin{figure}[hbt]
        \centering
        \includegraphics[width=0.25\linewidth]{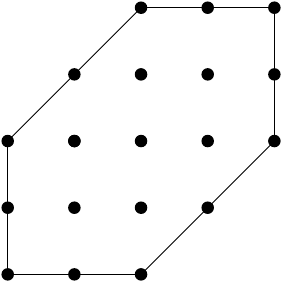}
        \caption{$P$ when $g=7$.}
\label{fig: d=4}
    \end{figure}

Recall \eqref{eq:cutRectangleBound}, which computes $\dim(\mathbb{M}_P)$ for a cut rectangle $P$. Observing that $P$ does not admit any column vectors, \eqref{eq:cutRectangleBound} yields
\[\dmodspace{P}=\left(\frac{g-1}{3}+3\right)\cdot5-9.\]
Next, we use a rounding argument: since $g\equiv 1\pmod 3$, there exists some $m\in\Z$ such that $g=3m+1$, and thus
\begin{align*}
    \dim(\mathbb{M}_P)
    &=5m+6
    =\left\lfloor 5m+\frac{20}{3}\right\rfloor\\
    &=\left\lfloor \frac{5g}{3}+5\right\rfloor
    =\lfloor U(g,4)\rfloor.
\end{align*}
 As $P$ has $g$ interior points, and $\egon(P)=\lw(P)=4$ (since $g\geq 7)$, $\dmodspace{P}$ provides a lower bound for $\dmndexpc{g}{4}$. Furthermore, $\floor{U(g,4)}$ provides an upper bound. Combining inequalities yields
$$\floor{U(g,4)}=\dmodspace{P}\leq\dmndexp{g}{4}\leq \floor{U(g,4)}.$$
Thus, we have equalities everywhere above.
\end{proof}

In the $d = 5$ case, we conjecture the following.

\begin{conjecture}\label{thm: d=5}
    Let $g \geq 12$. Then,
    \[\dmndexp{g}{5}=\floor{U(g,5)}- 1 = \floor{\frac{3g}{2}} + 6.\]
\end{conjecture}

We can supply the matching lower bound via the cut rectangles discussed in Lemma~\ref{lem:rPcPbound}. 

\begin{lemma}\label{lem:d=5LB}
    Let $g \geq 12$. Then,
    \[\dmndexp{g}{5}\geq\floor{U(g,5)}- 1 = \floor{\frac{3g}{2}} + 6.\]
\end{lemma}

\begin{proof}
    Recall \eqref{eq:cutRectangleBound}, which computes $\dim(\mathbb{M}_P)$ for a cut rectangle $P$. In each case we present, one may compute $\dim(\mathbb{M}_P) = \floor{U(g,5)} - 1$.
    \begin{itemize}
        \item \underline{Case 1: $g \equiv 0 \pmod{4}$}
        Let $P$ be the cut rectangle obtained via antipodal cuts of lengths $2, 3$ from a $d \times (\frac{g + 4}{4} +1)$ rectangle. 
        
        \item \underline{Case 2: $g \equiv 1 \pmod{4}.$}
        Let $P$ be the cut rectangle obtained via one cut of length $3$ from the $d \times (\frac{g+3}{4} + 1)$ rectangle. 
        
        \item \underline{Case 3: $g \equiv 2 \pmod{4}.$}
        Let $P$ be the cut rectangle obtained via antipodal cuts of length $2, 2$ from the $d \times (\frac{g+2}{4} + 1)$ rectangle. 
        
        \item \underline{Case 4: $g \equiv 3 \pmod{4}.$}
        Let $P$ be the cut rectangle obtained via one cut of length $2$ from the $d \times (\frac{g+1}{4} + 1)$ rectangle. 
    \end{itemize}
    
    Note that we require $g \geq 12$ to ensure that $P$ has expected gonality at least $5$.
    
\end{proof}

We conjecture that it is possible to strengthen the upper bound of $\floor{U(g,5)}$ to $\floor{U(g,5)} - 1$ via similar lattice geometry arguments.

\bibliographystyle{alpha}
\newcommand{\etalchar}[1]{$^{#1}$}

\newpage

\appendix

\section{Computing dimensions of moduli spaces for small genus}

Below, we show Theorem~\ref{thm:low_genus_computation} for all $g \leq 6$ and $g = 8$. Recall that the hyperelliptic $d = 2$ case is shown in Theorem~\ref{thm:d_is_2}, so we only enumerate the non-hyperelliptic polygons in the following proofs.

\label{sec: tropical_Appendix}

\subsection{Genus 0}

Recall the expected gonality of a genus $0$ polygon is $1$. By the tropical Riemann-Roch theorem~\cite{bn, GK08}, a metric graph has gonality $1$ if and only if it is a tree. Thus, $\mnd{0}{1} = \mndexp{0}{1}$.

\subsection{Genus 1 and 2}

If a polygon has genus 1 or 2, then $P$ is hyperelliptic and by Theorem~\ref{thm:d_is_2}, we have $\mnd{1}{2} = \mndexp{1}{2}$.

\subsection{Genus 3 and 4}

We claim that all non-hyperelliptic maximal polygons of genus $3$ and $4$ have expected gonality $3$. There is only one non-hyperelliptic maximal polygon of genus $3$ up to equivalence, namely $4\Sigma$. In genus $4$, there are three such polygons, namely those in Figure~\ref{figure:genus_4_polygons}. Recall that the left-most polygon is $2\Upsilon$, whose expected gonality is specially defined to be $3$. It can then be seen that all genus $4$ maximal non-hyperelliptic polygons have expected gonality $3$.

    \begin{figure}[hbt]
        \centering
        \includegraphics[width=1.0\linewidth]{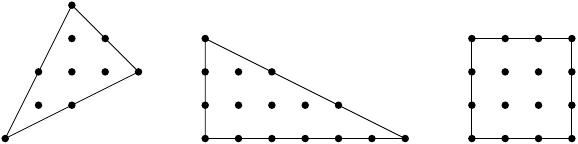}
        \caption{Maximal polygons of genus $4$ up to equivalence.}
        \label{figure:genus_4_polygons}
    \end{figure}
    
Since non-hyperelliptic polygons do not give rise to metric graphs of gonality $2$ (Theorem~\ref{thm:ralph_hyperelliptic}), we may invoke Lemma~\ref{lma:gon_at_most_egon_graphs} to find that these polygons only gives rise to metric graphs of gonality $3$. Thus, $\mnd{3}{3} \subseteq \mndexp{3}{3}$ and $\mnd{4}{3} \subseteq \mndexp {4}{3}$. The reverse inclusion comes from Theorem~\ref{thm:d_is_3}, so $\mnd{3}{3} = \mndexp{3}{3}$ and $\mnd{4}{3} = \mndexp{4}{3}$.

\subsection{Genus 5}

The proof for $g = 5$ can be found in the main text of the paper in Section~\ref{sec:dimlowg}.

\subsection{Genus 6}

We enumerate the maximal polygons of genus $6$ below in Figure~\ref{figure:genus_6_polygons}. Label them $P_1^{6}, P_2^{6}, P_3^{6}, P_4^{6}, P_5^{6}$ from left to right. Correspondingly, these polygons have expected gonality sequence $3, 3, 4, 4, 4$, and the dimension of their moduli spaces are $13, 12, 12, 13, 12$. Thus, $\dim (\mndexp{6}{3}) = \dim (\mndexp{6}{4}) = 13$. By Theorem~\ref{thm:d_is_3}, $\dim (\mnd{6}{3}) \geq 13$, and $\dim (\mnd{6}{3}) \leq 13$ because the ambient space $\mathbb{M}_{6}^\mathrm{nd}$ is $13$-dimensional.

    \begin{figure}[hbt]
        \centering
        \includegraphics[width=1.0\linewidth]{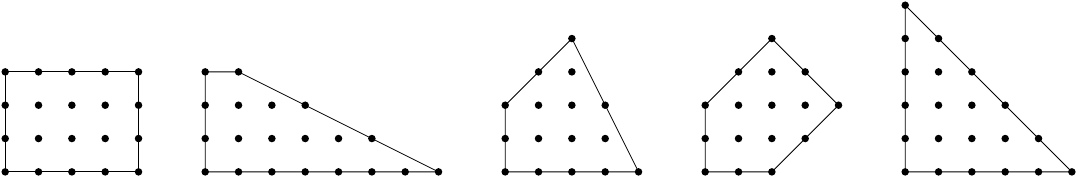}
        \caption{Maximal polygons of genus $6$ up to equivalence.}
        \label{figure:genus_6_polygons}
    \end{figure}

To show $\dim \mnd{6}{4} \geq 13$, we consider the family of metric graphs arising from the beehive triangulation $\Delta$ of $P_4^{6}$ shown in Figure~\ref{fig:genus6_triangulation}. The eggs of an order $4$ scramble on the combinatorial graph dual to $\Delta$ are also shown. Thus, $\dim (\mnd{6}{4}) = 13$.

    \begin{figure}[hbt]
        \centering
        \includegraphics[width=0.6\linewidth]{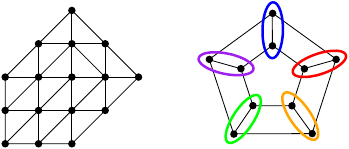}
        \caption{The triangulation $\Delta$ on $P_4^{6}$ and the dual graph with eggs circled.}
        \label{fig:genus6_triangulation}
    \end{figure}

\subsection{Genus 8}

We enumerate the maximal polygons of genus $8$ below in Figure~\ref{figure:genus_8_polygons}. Label the polygons $P_1^{8}, \dots, P_{10}^{8}$ sequentially by rows, from left to right then top to bottom. Correspondingly, these polygons have expected gonality sequence $4, 4, 4, 3, 4, 4, 4, 4, 4, 4$, and the dimension of their moduli spaces are $16, 15, 15, 17, 15, 11, 17, 14, 15, 14$. Thus, $\dim(\mndexp{8}{3}) = \dim(\mndexp{8}{4}) = 17$. By Theorem~\ref{thm:d_is_3}, $\dim (\mnd{8}{3}) \geq 17$, and $\dim (\mnd{8}{3}) \leq 17$ because the ambient space $\modspace{8}^{\text{nd}}$ is $17$-dimensional.
    \begin{figure}[hbt]
        \centering
        \includegraphics[width=0.8\linewidth]{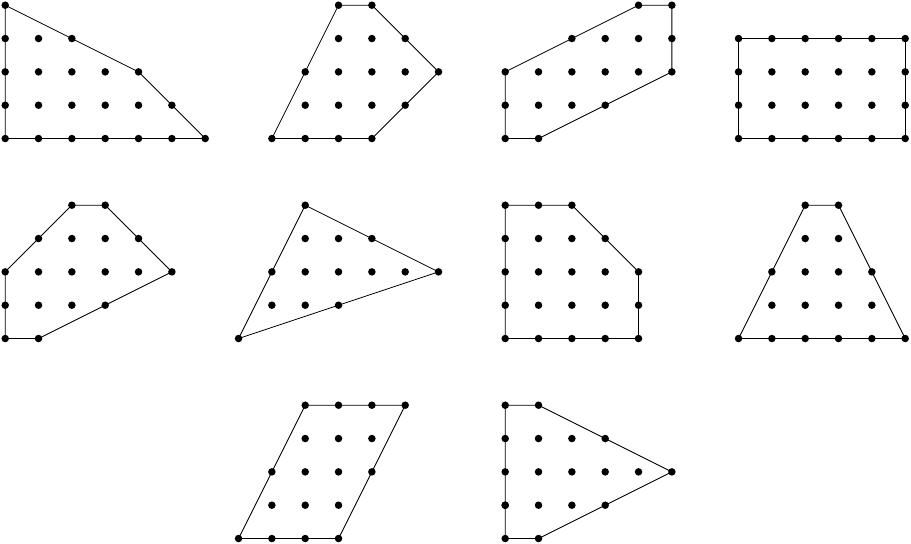}
        \caption{Maximal polygons of genus 8 up to equivalence.}
        \label{figure:genus_8_polygons}
    \end{figure}

To show $\dim \mnd{8}{4} \geq 17$, we consider the family of metric graphs arising from the beehive triangulation $\Delta$ of $P_7^{8}$ shown in Figure~\ref{figure:genus_8_triangulation}. The eggs of an order $4$ scramble on the combinatorial graph dual to $\Delta$ is also shown. Thus, $\dim (\mnd{8}{4}) = 17$.

    \begin{figure}[hbt]
        \centering
        \includegraphics[width=0.7\linewidth]{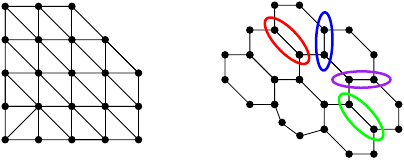}
        \caption{The triangulation $\Delta$ on $P_7^{8}$ and the dual graph with eggs circled.}
        \label{figure:genus_8_triangulation}
    \end{figure}

\end{document}